\documentclass[11pt, twoside]{amsart}

\title[Constructing non-semisimple modular categories]{Constructing non-semisimple modular categories with 
relative monoidal centers}

\author{Robert Laugwitz}
\address{School of Mathematical Sciences,
University of Nottingham, University Park, Nottingham, NG7 2RD, UK}
\email{robert.laugwitz@nottingham.ac.uk}

\author{Chelsea Walton}
\address{Department of Mathematics, Rice University,
P.O. Box 1892, Houston, TX 77005-1892, USA}
\email{notlaw@rice.edu}

\usepackage{mathabx}
\usepackage{import}
\usepackage{amsmath}
\usepackage{amsfonts}
\usepackage{amsthm}
\usepackage{amssymb,bbm}
\usepackage{dsfont}
\usepackage[alphabetic, initials]{amsrefs}
\usepackage[english]{babel}
\usepackage{url}
\usepackage{fancyhdr}
\usepackage{graphicx}
\usepackage{verbatim}
\usepackage[normalem]{ulem}
\newcommand{\stkout}[1]{\ifmmode\text{\sout{\ensuremath{#1}}}\else\sout{#1}\fi}
\usepackage[
colorlinks=true,
linkcolor=black, 
anchorcolor=black,
citecolor=black,
urlcolor=black, 
]{hyperref}
\usepackage{geometry}
\usepackage{microtype} 
\usepackage[dvipsnames]{xcolor}

\input xy
\xyoption{all}

\definecolor{forest}{rgb}{0.0, 0.5, 0.0}

\newcommand \red {\textcolor{red}}

\newcommand{\bijar}[1][]{%
 \ar[#1]
 \ar@<0.7ex>@{}[#1]|-*=0[@]{\sim}} 

\usepackage{enumitem}

\usepackage{calrsfs}
\DeclareMathAlphabet{\cal}{OMS}{zplm}{m}{n}

\usepackage[justification=centering]{caption}

\newcommand{\leftexpsub}[3]{{\vphantom{#3}}^{#1}_{#2}{#3}}

\newcommand{\lYD}[1]{\leftexpsub{#1}{#1}{\mathsf{YD}}}

\newcommand{\Set}[1]{\left\lbrace #1\right\rbrace}

\newcommand{\oop}{\mathrm{op}}
\newcommand{\cop}{\mathrm{cop}}
\newcommand{\ov}[1]{\overline{#1}}

\newcommand{\lmod}[1]{#1\text{-}\mathsf{mod}}

\newcommand{\rmod}[1]{\mathsf{mod}\text{-}#1}

\newcommand{\lcomod}[1]{#1\text{-}\mathsf{comod}}
\newcommand{\rcomod}[1]{\mathsf{comod}\text{-}#1}

\newcommand{\coev}{\mathsf{coev}}
\newcommand{\coevr}{\widetilde{\mathsf{coev}}}

\newcommand{\Drin}{\operatorname{Drin}}
\newcommand{\ev}{\mathsf{ev}}
\newcommand{\evr}{\widetilde{\mathsf{ev}}}
\newcommand{\End}{\operatorname{End}}

\newcommand{\Hom}{\mathsf{Hom}}
\newcommand{\iHom}{\underline{\mathsf{Hom}}}

\newcommand{\ide}{\mathsf{Id}}

\newcommand{\isomorph}{\stackrel{\sim}{\to}}

\newcommand{\Mat}{\operatorname{Mat}}

\newcommand{\Ob}{\mathsf{Ob}}
\newcommand{\one}{\mathds{1}}

\newcommand{\triv}{\mathrm{triv}}

\newcommand{\Alg}{\mathsf{Alg}}

\newcommand{\Bialg}{\mathsf{Bialg}}

\newcommand{\Coalg}{\mathsf{Coalg}}

\newcommand{\FPdim}{\mathsf{FPdim}}

\newcommand{\HopfAlg}{\mathsf{HopfAlg}}

\newcommand{\Vect}{\mathsf{vect}_\Bbbk}

\newcommand{\BB}{\mathfrak{B}}
\newcommand{\TT}{\mathfrak{T}}
\newcommand{\II}{\mathfrak{I}}

\providecommand{\fr}[1]{\mathfrak{#1}}

\newcommand{\mZ}{\mathbb{Z}}

\newcommand{\cA}{\cal{A}}
\newcommand{\cC}{\cal{C}}
\newcommand{\cD}{\cal{D}}
\newcommand{\cB}{\cal{B}}
\newcommand{\cE}{\cal{E}}
\newcommand{\cF}{\cal{F}}

\newcommand{\cT}{\cal{T}}

\newcommand{\cM}{\cal{M}}

\newcommand{\cS}{\cal{S}}

\newcommand{\cZ}{\cal{Z}}

\newcommand{\sfC}{\mathsf{C}}

\newcommand{\rT}{\mathrm{T}}

\newcommand{\qs}{\mathbf{q}}

\newtheoremstyle{defstyle}
  {0.5cm}                   
  {0.5cm}                   
  {\normalfont}           
  {}     
  {\normalfont\bfseries}  
  {:}                     
  {0.3cm}              
  {\thmname{#1}\thmnumber{ #2}\thmnote{ (#3)}}

\numberwithin{equation}{section}

\newtheorem*{rep@theorem}{\rep@title}
\newcommand{\newreptheorem}[2]{%
\newenvironment{rep#1}[1]{%
 \def\rep@title{#2 \ref{##1}}%
 \begin{rep@theorem}}%
 {\end{rep@theorem}}}
\makeatother

\newtheorem{theorem}{Theorem}[section]

\newtheorem{proposition}[theorem]{Proposition}
\newreptheorem{proposition}{Proposition}
\newtheorem{corollary}[theorem]{Corollary}
\newreptheorem{corollary}{Corollary}
\newtheorem{lemma}[theorem]{Lemma}

\newtheorem{theorem*}{Theorem}
\newreptheorem{theorem}{Theorem}

\theoremstyle{definition}
\newtheorem{definition}[theorem]{Definition}
\newtheorem{notation}[theorem]{Notation}

\newtheorem{example}[theorem]{Example}
\newtheorem{remark}[theorem]{Remark}

\newtheorem{question}[theorem]{Question}

\makeatletter              
\let\c@equation\c@theorem  
\makeatother
\numberwithin{equation}{section}

\makeatletter
\@namedef{subjclassname@2020}{%
  \textup{2020} Mathematics Subject Classification}
\makeatother

\subjclass[2020]{18M20, 18M15, 17B37}
\keywords{braided Drinfeld double, modular tensor category, M\"{u}ger centralizer, Nichols algebra relative monoidal center, small quantum group}

\begin{document}

\maketitle

\begin{abstract}
This paper is a contribution to the construction of non-semisimple modular categories. We establish when M\"uger centralizers inside non-semisimple modular categories are also modular. 
 As a consequence, we obtain conditions under which relative monoidal centers give (non-semisimple) modular categories, and we also show that examples include representation categories of small quantum groups. We further derive conditions under which representations of more general quantum groups, braided Drinfeld doubles of Nichols algebras of diagonal type, give (non-semisimple) modular categories.
\end{abstract}



\section{Introduction}\label{sec:intro}
The purpose of this article is to establish new constructions of modular tensor categories in the non-semisimple setting. We work over an algebraically closed field $\Bbbk$.

\smallskip

To begin, let us recall the main structure of interest in this work, which is due to Kerler--Lyubashenko \cite{KL}. We refer the reader to Section~\ref{sec:monoidal} for a discussion of various types of tensor categories relevant here. Take $\Vect$ to be the tensor category of finite-dimensional $\Bbbk$-vector spaces, and for a braided tensor category $\cC$, let $\cC'$ be the {\it M\"{u}ger center} of $\cC$ (see \eqref{eq:Mugercenter}). 

\begin{definition}[Definitions~\ref{def:nondeg}, \ref{def:modular}]
Take $\cC$ a braided finite tensor category. We call $\cC$ a {\it modular tensor category (MTC)} if $\cC$ is non-degenerate (i.e., $\cC' \simeq \Vect$) and ribbon.
\end{definition}

Note that the definition above does not require semisimplicity, as the commonly used definition of an MTC (see, e.g., \cite{BK}). MTCs provide actions of the modular group though their \emph{modular data}, the $S$- and $T$-matrices, a structure that emerged from mathematical physics \cite{MoS}. Semisimple MTCs have appeared in various fields such as low-dimensional topology \cite{Turaev-1992}, conformal field theory \cite{MoS,Hua,Gannon}, and subfactor theory \cite{KLM-2001}; they have been under intense investigation towards classification results by rank (see, e.g., \cite{RSW}). 

The definition of a non-semisimple MTC of \cite{KL} has been given further justification through equivalent characterizations in \cite{Shi1}. Moreover, non-semisimple MTCs are gaining traction due to their growing list of applications, starting with non-semisimple topological quantum field theories \cite{KL}, most recently in \cite{DGGPR}, to the study of logarithmic conformal field theories \cite{HLZ}, modular functors \cites{FSS}, and mapping class group actions \cites{LMSS}. Some module categories of small quantum groups (and of related quasi-Hopf algebras) have been shown to yield examples of non-semisimple MTCs \cites{Lusztig,GLO,LO, Negron}. But, in general, non-semisimple MTCs are not well-understood via classification nor examples, and we aim to contribute to the latter in this work.


\smallskip

One of the main examples of an MTC in the semisimple setting is the monoidal center $\cZ(\cC)$ [Section~\ref{sec:braided}] of a trace-spherical tensor category $\cC$ [Section~\ref{sec:rigid}] (often referred to as a spherical category) \cite{BW}. We discuss in Section~\ref{sec:spherical} a non-semisimple generalization of this result, due to Shimizu, which we will employ later in  our main results. 
First, we recall the  set of square roots  of the Radford isomorphism of a finite tensor category, denoted by $\mathsf{Sqrt}_\cC(D,\xi_{D})$ here [Definition~\ref{def:sqrtset}] that Shimizu uses to parameterize  ribbon structures for  $\cZ(\cC)$ [Theorem~\ref{thm:Shimizuribbon}]. This recovers a result Kauffman--Radford for the ribbonality of the Drinfeld double [Theorem~\ref{thm:KR}]; see Proposition~\ref{prop:KRvsShi} for the explicit connection between these theorems. Next, we recall Douglas--Schommer-Pries--Snyder's notion of sphericality \cite{DSS} in the non-semisimple setting [Definition~\ref{def:spherical}].
By \cite{Shi2}, this notion of sphericality implies that $\mathsf{Sqrt}_\cC(D,\xi_{D})$ is non-empty [Remark~\ref{rem:sph-rib}]. We provide equivalent conditions for the representation category of a Hopf algebra to be spherical [Proposition~\ref{prop:Hspherical}]. Moreover, we obtain the following result.

\begin{proposition}[Proposition~\ref{prop:ribbonspherical}]
Any unimodular finite ribbon category is spherical in the sense of \cite{DSS}.
 \end{proposition}

Shimizu's results from \cite{Shi1} on the modularity of $\cZ(\cC)$ are recalled in Theorem~\ref{thm:Shimizucenter} and Corollary~\ref{cor:Shimizucenter}--- the hypotheses are, respectively, that $\mathsf{Sqrt}_\cC(D,\xi_{D})$ is non-empty and  that $\cC$ is spherical (as in \cite{DSS}).

\smallskip

We build on the results on $\cZ(\cC)$ above to obtain Theorem~\ref{thm:ZBCmodular-intro} below on the modularity of relative monoidal centers. Before this, we establish a general result on the modularity of M\"uger centralizer, Theorem~\ref{thm:centmodular-intro}. These are the main results of the paper, presented in Section~\ref{sec:ZBC-modular}.

\smallskip

To proceed, note that a full subcategory of a category is called \emph{topologizing} if it is closed under finite direct sums and subquotients \cites{Ros, Shi1}. Moreover, recall the notion of the  {\it M\"{u}ger centralizer} of a subset of objects in a braided category; see \eqref{eq:centralizer}.

\begin{theorem}[Theorem~\ref{thm:centmodular}] \label{thm:centmodular-intro}
Let $\cD$ be a modular category, let $\cE$ be a topologizing braided tensor subcategory of $\cD$, and consider the  M\"{u}ger centralizer ${\sf C}_{\cD}(\cE)$. Then,  $$\sfC_{\cD}(\cE)' \simeq \cE'.$$ As a consequence, $\sfC_{\cD}(\cE)$ is modular if and only if $\cE$ is modular.
\end{theorem}

This generalizes a result of M\"{u}ger in the semisimple case \cite{Mue}*{Corollary~3.5}.

\smallskip

Next we discuss a special case of M\"{u}ger centralizers: relative monoidal centers. Take a braided category $(\cB, \;\psi_{X,Y}\colon X \otimes Y \overset{\sim}{\to} Y \otimes X)$  with braided opposite $\overline{\cB}:=(\cB, \;\psi^{-1}_{Y,X}\colon X \otimes Y \overset{\sim}{\to} Y \otimes X)$. We say that a monoidal category $\cC$ is {\it $\cB$-central} if there exists a faithful braided monoidal functor $G\colon  \overline{\cB} \to \cZ(\cC)$ [Definition~\ref{def:Bcentral}]. With such a category $\cC$, one can form the {\it relative monoidal center} $\cZ_{\cB}(\cC)$, which is a full subcategory of $\cZ(\cC)$ consisting of objects that centralize (via the braiding of $\cZ(\cC)$) all objects in the image of $G$ [Definition~\ref{def:relcenter}]. In fact, $\cZ_{\cB}(\cC)$ is equal to the M\"{u}ger centralizer ${\sf C}_{\cZ(\cC)}(G(\overline{\cB}))$ [Remark~\ref{rem:ZBC}]. Given Theorem~\ref{thm:centmodular-intro}, we obtain the following result.

\begin{theorem} 
[Theorem~\ref{thm:ZBCmodular}] \label{thm:ZBCmodular-intro}
 Let $\cB$ be a non-degenerate braided finite tensor category, and $\cC$ a $\cB$-central 
 finite tensor category so that the set   $\mathsf{Sqrt}_\cC(D,\xi_{D})$ from Definition \ref{def:sqrtset} is non-empty. If the full image $G(\ov{\cB})$ is a topologizing subcategory of $\cZ(\cC)$, then the relative monoidal center $\cZ_\cB(\cC)$ is modular.
\end{theorem}

For $\cB$ a rigid braided category and $H$ a Hopf algebra in $\cB$, Theorem~\ref{thm:ZBCmodular-intro} can be used to study the modularity of the category of finite-dimensional $H$-Yetter Drinfeld modules in $\cB$ [Example~\ref{expl:br-YD}]. If, further, $\cB$ is a representation category of a quasi-triangular Hopf algebra $K$, then Theorem~\ref{thm:ZBCmodular-intro} can also be used to study the modularity of the representation category of the braided Drinfeld double $\Drin_K(H,H^*)$ [Example~\ref{expl:br-Drin}]. From this, we show, as a first example, that the representation category of the small quantum group  $u_q(\mathfrak{sl}_2)$, for $q$ a root of unity of odd order, is an MTC [Example~\ref{ex:uqsl2}]. Generalizations of this (non-semisimple) MTC will be given in Proposition~\ref{cor:UBqmod-intro} below.

\smallskip

Motivated again by work of M\"{u}ger in the semisimple  case,  we next consider the decomposition of modular tensor categories into Deligne tensor product of modular subcategories. We obtain the result below; cf., \cite{Mue}*{Theorem~4.2}.

\begin{theorem} [Theorem~\ref{thm:decomp}]
Let $\cD$ be a modular tensor category, with a topologizing non-degenerate braided tensor subcategory $\cE$. Then, there is an equivalence of ribbon categories:
$$\cD \simeq \cE \boxtimes {\mathsf C}_\cD(\cE).$$
\end{theorem}

In particular, under the conditions of Theorem \ref{thm:ZBCmodular-intro}, the relative monoidal center is related to the monoidal center through the factorization $$\cZ(\cC)\simeq \ov{\cB}\boxtimes \cZ_\cB(\cC).$$
Continuing an example mentioned above, for $H$ a Hopf algebra in the braided tensor category $\lmod{K}$, we have that $\lmod{\Drin(H \rtimes K)} \simeq \lmod{K} \;\boxtimes \; \lmod{\Drin_K(H,H^*)}$ as modular categories under the hypotheses of Theorem~\ref{thm:ZBCmodular-intro}; see Example~\ref{ex:decomp1}(2).

\smallskip

As in \cite{Mue}, we call a modular tensor category $\cC$ in the non-semisimple setting  {\it prime} if every topologizing non-degenerate braided tensor subcategory is equivalent to either $\cC$ or $\Vect$. We obtain the result below as an immediate consequence of the theorem above, cf. \cite{Mue}*{Theorem~4.5}.

\begin{corollary}[Corollary~\ref{cor:primedecomp}]
Every modular tensor category is equivalent to a finite Deligne tensor product of prime modular categories.
\end{corollary}

Although primality is difficult to detect in the semisimple case (see \cite{Mue}*{Section~4}), we inquire when it holds in the non-semisimple case, particularly for  $\lmod{\Drin_K(H,H^*)}$ in the example above; see Question~\ref{ques:DKH}.

\smallskip

Finally, we construct several examples of non-semisimple MTCs, via Theorem~\ref{thm:ZBCmodular-intro}, by using   Nichols algebras of diagonal type in braided categories of comodules over finite abelian groups.

\begin{proposition}[Proposition~\ref{cor:UBqmod}] \label{cor:UBqmod-intro}
Take $K:=\Bbbk G$, for $G$ a finite abelian group, assume that $\operatorname{char}\Bbbk=0$, and take $\BB$ a finite-dimensional Nichols algebra of diagonal type in a certain braided category $\cB$ of $K$-comodules. Consider the relative monoidal center,  $\cD:=\cZ_{\cB}(\lmod{\BB}(\cB))$,  or equivalently the category of finite-dimensional modules over the braided Drinfeld double $\Drin_{K^*}(\BB^*,\BB)$. 

Then, $\cD$ is modular when \textnormal{(i)} the canonical symmetric bilinear form $b$ on the coquasi-triangular Hopf algebra $K$ is non-degenerate, and \textnormal{(ii)} certain conditions involving elements of the top degree of $\BB$ and on the dual R-matrix of $K$ are satisfied.
\end{proposition}

Note that the Drinfeld double of the bosonization of Nichols algebras has been studied in the literature, see e.g. \cite{Hec4}, where two copies of the group algebra consitute the Cartan part. In this paper, an approach is used where the Cartan part consists of a simple (dual) group algebra $K^*$.

We end the paper by constructing, via Proposition~\ref{cor:UBqmod-intro},  examples of non-semisimple modular tensor categories attached to Nichols algebras of Cartan type [Example~\ref{ex:uqg}] and not of Cartan type [Example~\ref{expl-new}]. The former includes the representation category of the small quantum group $u_q(\mathfrak{g})$ at an odd root of unity.
Thus, the methods developed in this paper provide an alternative argument showing that the category of finite-dimensional $u_q(\fr{g})$-modules is a non-semisimple MTC, which was previously obtained in \cite{Lyu}*{Section A.3}. See also \cites{LO,GLO} for more general results on the modularity of representation categories of small quantum groups. On the other hand, the non-semisimple MTCs in Example~\ref{expl-new} illustrate that our methods can be used to analyze the modularity of representation categories attached to a broader class of Nichols algebras beyond small quantum groups.


\section{Preliminaries on Monoidal Categories}\label{sec:monoidal}
 In this section, we review terminology pertaining to monoidal categories. We refer the reader to \cite{BK}, \cite{EGNO}, and \cite{TV} for general information. We recall monoidal categories and module categories [Section~\ref{sec:monoidal1}], various types of rigid categories [Section~\ref{sec:rigid}], finite tensor categories [Section~\ref{sec:finitetens}], various braided monoidal categories [Section~\ref{sec:braided}], algebraic structures in finite tensor categories [Section~\ref{sec:algstr}], ribbon monoidal categories [Section~\ref{sec:ribbon}], and modular tensor categories in the non-semisimple setting [Section~\ref{sec:modular}].
 
 \smallskip
 
We assume that all categories here are {\it locally small} (i.e., the collection of morphisms between any two objects is a set),  and that all categories here are abelian. A full subcategory of a category is called \emph{topologizing} if it is closed under finite direct sums and subquotients \cite{Ros}*{Section~3.5.3}, \cite{Shi1}*{Definition~4.3}. 
 Given a functor $F\colon \cC\to \cD$ between two categories $\cC$ and $\cD$, the \emph{full image} of $F$ is the full subcategory of $\cD$ on all objects isomorphic to an object of the form $F(C)$ for $C$ in~$\cC$. 


\subsection{Monoidal categories, monoidal functors, and module categories}
\label{sec:monoidal1}
We refer the reader to \cite{EGNO}*{Sections~2.1--2.6 and~7.1} and \cite{TV}*{Sections~1.1--1.4} for further details of the items discussed here.

\smallskip

A {\it monoidal category} consists of a category $\cC$ equipped with a bifunctor $\otimes\colon  \cC \times \cC \to \cC$, a natural isomorphism $\alpha_{X,Y,Z}\colon  (X \otimes Y) \otimes Z \overset{\sim}{\to} X \otimes (Y \otimes Z)$ for each $X,Y,Z \in \cC$, an object $\one \in \cC$, and natural isomorphisms $l_X\colon  \one \otimes X \overset{\sim}{\to} X$ and $r_X\colon  X  \otimes \one \overset{\sim}{\to} X$ for each $X \in \cC$, such that the pentagon and triangle axioms hold. By MacLane's coherence theorem, we will assume that all monoidal categories are {\it strict} in the sense that $(X \otimes Y) \otimes Z = X \otimes (Y \otimes Z)$ and $\one \otimes X  = X = X \otimes \one$, for all $X, Y, Z \in \cC$; that is, $\alpha_{X,Y,Z},\; l_X,\; r_X$ are identity maps.  For a monoidal category $(\cC, \otimes, \one)$, define the {\it opposite
monoidal category} to  
$\cC^{\otimes \text{op}} = (\cC, ~\otimes^{\text{op}}, ~\one)$
where 
$X \otimes^{\text{op}} Y := Y \otimes X$. 

\smallskip

A {\it (strong) monoidal functor} $(F, F_{-,-}, F_0)$ between monoidal categories $(\cC, \otimes_\cC, \one_\cC)$ to $(\cD, \otimes_\cD, \one_\cD)$ is a functor $F\colon  \cC \to \cD$ equipped with a natural isomorphism $F_{X,Y}\colon  F(X) \otimes_\cD F(Y) \isomorph F(X \otimes_\cC Y)$ for all $X,Y \in \cC$, and an isomorphism $F_0\colon  \one_\cC \isomorph F(\one_\cC)$ in $\cD$, that satisfy associativity and unitality constraints. An {\it equivalence of  monoidal categories} is provided by a  monoidal functor between the two monoidal categories that yields an equivalence of the underlying categories.

\smallskip

Representations of monoidal categories are provided by the next notion. A left {\it $\cC$-module category} is a category $\cM$ equipped with a bifunctor $\otimes\colon  \cC \times \cM \to \cM$, natural isomorphisms for associativity $m_{X,Y,M}\colon (X \otimes Y)\otimes M \to X\otimes(Y\otimes M)$, for all $X,Y \in \cC, M \in \cM$ satisfying the pentagon axiom, and
for each $M \in \cM$ a natural isomorphism $\one \otimes M \to M$ satisfying the triangle axiom.


\subsection{Rigid, pivotal, and trace-spherical monoidal categories} \label{sec:rigid}
We refer  to \cite{EGNO}*{Sections~2.10 and~4.7} and \cite{TV}*{Sections~1.5--1.7} for further details of the items discussed below.

\smallskip

A monoidal category $(\cC, \otimes, \one)$ is {\it rigid} if it comes equipped with left and right dual objects,  i.e., for each $X \in \cC$ there exist, respectively, an object $X^* \in \cC$ with co/evaluation maps $\ev_X\colon  X^* \otimes X \to \one$ and $\coev_X\colon  \one \to  X \otimes X^*$, as well as an object ${}^*X \in \cC$ with co/evaluation maps $\evr_X\colon  X \otimes {}^*X \to \one$, $\coevr_X\colon \one \to  {}^*X \otimes X$,  satisfying the usual coherence conditions of left and right duals. 
An object $X$ in a rigid monoidal category $\cC$ is {\it invertible} if $\ev_X$ and $\coev_X$ are isomorphisms.

\smallskip

 A rigid monoidal category is {\it pivotal} if it is equipped with isomorphisms $j_X\colon  X \overset{\sim}{\to} X^{**}$ natural in $X$ and satisfying $j_{X \otimes Y} = j_X \otimes j_Y$ for all $X,Y \in \cC$. 
Equivalently, a pivotal category is a rigid monoidal category such that the functors of left and right duality coincide as monoidal functors \cite{TV}*{Section 1.7}.

\smallskip

The {\it quantum dimension} of an object $X$ of a pivotal (rigid) monoidal category $(\cC, \otimes, \one, j)$  is defined to be $\dim_j(X) = {\sf ev}_{X^*}  (j_X \otimes {\sf Id}_{X^*})  {\sf coev}_X \in \End_\cC(\one).$ 
A pivotal monoidal category $(\cC, \otimes, \one, j)$ is {\it trace-spherical} if $\dim_j(X) = \dim_j(X^*)$ for each $X \in \cC$.


\subsection{Finite tensor categories} \label{sec:finitetens}
Recall that $\Bbbk$ is an algebraically closed field. We now discuss certain $\Bbbk$-linear monoidal categories following the terminologies of \cite{EGNO}*{Sections~1.8,~7.1--7.3,~7.9}.

\smallskip

A $\Bbbk$-linear abelian category $\cC$ is {\it locally finite} if, for any two objects $V,W$ in $\cC$, $\Hom_{\cC}(V,W)$ is a finite-dimensional $\Bbbk$-vector space and every object has a finite filtration by simple objects. Moreover, we say that $\cC$ is {\it finite} if there are finitely many isomorphism classes of simple objects. Equivalently, $\cC$ is locally finite if it is equivalent to the category of finite-dimensional comodules over a $\Bbbk$-coalgebra (or, to modules over a finite-dimensional $\Bbbk$-algebra if $\cC$ is finite).
A  {\it tensor category} is a locally finite, rigid, monoidal category $(\cC, \otimes, \one)$ such that $\otimes$ is $\Bbbk$-linear in each slot and $\one$ is a simple object of $\cC$.
 A {\it tensor functor} is a $\Bbbk$-linear exact  monoidal functor  between  tensor categories. 
 
\smallskip

An example of a finite tensor category is $\Vect$, the category of finite-dimensional $\Bbbk$-vector spaces. More generally, the category $\lmod{H}$ of finite-dimensional $\Bbbk$-modules over a (finite-dimensional) Hopf algebra $H$ is a (finite) tensor category.

\smallskip

We will use the following tensor product of finite tensor categories. The {\it Deligne tensor product} of two finite abelian categories is the abelian category $\cC \boxtimes \cD$  equipped with a bifunctor $\boxtimes\colon  \cC \times \cD \to \cC \boxtimes \cD$, $(X,Y) \mapsto X \boxtimes Y$, right exact in both variables so that for any abelian category $\cA$ and any bifunctor $F\colon  \cC \times \cD \to \cA$ right exact in both slots, there exists a unique right exact functor $\overline{F}\colon  \cC \boxtimes \cD \to \cA$ with $\overline{F} \circ \boxtimes = F$ \cite{Del}*{Section 5}. It is monoidal  when both $\cC$ and $\cD$ are so, via 
\begin{equation} \label{eq:Deligne-monoidal}
    (X \boxtimes Y) \otimes^{\cC \boxtimes \cD} (X' \boxtimes Y'):=(X \otimes^\cC X') \boxtimes (Y \otimes^\cD Y'),
\end{equation}
for all $X,X' \in \cC$ and $Y,Y' \in \cD$,
and with the unit object $\one_\cC \boxtimes \one_\cD$.
If $\cC$, $\cD$ are finite tensor categories, then so is $\cC\boxtimes \cD$. Given two tensor functors $F\colon \cC\to\cD$ and $F'\colon \cC'\to \cD'$ between finite tensor categories, there exists an induced tensor functor $F\boxtimes F'\colon \cC\boxtimes\cC'\to \cD\boxtimes \cD'$.
\smallskip

For a tensor category $\cC$ over $\Bbbk$, a left {\it module category} over $\cC$ is a module category $\cM$ as in Section~\ref{sec:monoidal1} with the requirement that $\cM$ is also $\Bbbk$-linear and abelian so that the underlying bifunctor is $\Bbbk$-linear on morphisms and exact in the first variable (it is always exact in the second variable).

\smallskip

An {\it internal Hom object} for a module category $\cM$ over a $\Bbbk$-linear, finite, tensor  category $\cC$ is an object $\iHom(M_1,M_2)$ in $\cC$, for $M_1, M_2 \in \cM$, that represents
the left exact functor $\cC \to \Vect,$ defined by $X \mapsto \Hom_{\cM}(X \otimes M_1, M_2).$
Namely, we have a natural isomorphism: $\Hom_{\cM}(X \otimes M_1, M_2) \cong \Hom_{\cC}(X,\iHom(M_1,M_2))$.


\subsection{Braided monoidal categories, the monoidal center \texorpdfstring{$\cZ(\cC)$}{Z(C)}, and the M\"uger center~\texorpdfstring{$\cC'$}{C'}} 
\label{sec:braided}
Here, we discuss braided tensor categories and related constructions, and refer the reader to \cite{BK}*{Chapter~1}, \cite{EGNO}*{Sections~8.1--8.3, 8.5,~8.20}, and \cite{TV}*{Sections~3.1 and~5.1} for more information.

\smallskip

A {\it braided tensor category} $(\cC, \otimes, \one, c)$ is a tensor category equipped with a natural isomorphism $c_{X,Y}\colon  X \otimes Y \overset{\sim}{\to} Y \otimes X$ for each $X,Y \in \cC$ such that the hexagon axioms hold. 
By a {\it braided tensor subcategory} of a braided tensor category $\cC$ we mean a subcategory of $\cC$ containing the unit object of $\cC$, closed under the tensor product of $\cC$, and containing the braiding isomorphisms. 
A {\it braided tensor functor} between braided tensor categories $\cC$ and $\cD$ is a tensor  functor $(F, F_{*,*},F_0)\colon  \cC \to \cD$ so that $F_{Y,X} \; c^{\cD}_{F(X),F(Y)} = F(c^{\hspace{.01in}\cC}_{X,Y}) \; F_{X,Y}$ for all $X,Y \in \cC$. An {\it equivalence of  braided tensor categories} is a braided tensor functor between the two tensor categories that yields an equivalence of the underlying categories.

\smallskip 

An important example of a braided tensor category is the {\it monoidal center} (or {\it Drinfeld center}) $\cZ(\cC)$ of a tensor category $(\cC, \otimes, \one)$: its objects are pairs $(V, c_{V,-})$ where $V$ is an object of $\cC$ and $c_{V,X}\colon  V \otimes X \overset{\sim}{\to} X \otimes V$ is a natural isomorphism (called a  {\it half-braiding}) satisfying $c_{V,X \otimes Y} = (\ide_{X} \otimes c_{V,Y})(c_{V,X} \otimes \ide_{Y}).$
An important feature of $\cZ(\cC)$ is the braiding defined by $$c_{(V,c_{V,-}),(W,c_{W,-})}:=c_{V,W}\colon V\otimes W \isomorph W\otimes V.$$



\begin{proposition}[see \cite{EGNO}*{Section~7.13}] \label{prop:ZC-finitetens}
If $\cC$ is a (finite) tensor category, then $\cZ(\cC)$ is a braided (finite) tensor category. \qed
\end{proposition}

Given two braided finite tensor  categories $(\cC,\otimes^\cC,\one_\cC,c^\cC)$ and $(\cD,\otimes^\cD,\one_\cD,c^\cD)$, the Deligne tensor product $\cC\boxtimes \cD$  is a braided finite tensor category. The braiding is obtained from 
\begin{equation} \label{eq:Deligne-braiding}
    c_{X\boxtimes Y,X'\boxtimes Y'}=c^\cC_{X,X'}\boxtimes c^\cD_{Y,Y'}\colon (X\otimes^\cC X')\boxtimes (Y\otimes^\cD Y')\to (X'\otimes^\cC X)\boxtimes (Y'\otimes^\cD Y)
\end{equation}
for all $X,X'\in \cC$, $Y,Y'\in\cD$.

\smallskip

We need to consider later the {\it M\"uger center}  of a braided tensor category $(\cC, \otimes, \one, c)$, which is the full subcategory on the objects
\begin{equation} \label{eq:Mugercenter}
\Ob (\cC') := \{ X \in \cC ~|~ c_{Y,X}\; c_{X,Y} = {\sf Id}_{X \otimes Y} \text{ for all } Y \in \cC \}.
\end{equation}


\subsection{Algebraic structures in tensor categories}
\label{sec:algstr}

In this section, let $\cC:=(\cC, \otimes, \one)$ be a tensor category over $\Bbbk$. Assume that all structures below are $\Bbbk$-linear as well.


\subsubsection{(Co)algebras and their (co)modules} \label{sec:co/alg}
 We discuss in this part algebras and coalgebras in~$\cC$ and their (co)modules. More information is available in \cite{EGNO}*{Section~7.8} and \cite{TV}*{Section~6.1}.

\smallskip

An {\it algebra} in $\cC$ is an object $A \in \cC$ equipped with two morphisms $m\colon  A \otimes A \to A$ (multiplication) and $u\colon  \one \to A$ (unit) satisfying $m(m\otimes \ide_A) = m(\ide_A \otimes m)$ and $m(u \otimes \ide_A) = \ide_A = m(\ide_A \otimes  u)$. We denote by $\Alg(\cC)$ the category of algebras in $\cC$, where morphisms in $\Alg(\cC)$ are morphisms $f\colon  A \to A'$ in $\cC$ so that $f \; m_A = m_{A'} (f \otimes f)$ and $f \; u_A = u_{A'}$. 

\smallskip

Given an algebra $A$ in $\cC$, a \emph{left $A$-module in $\cC$} is a pair $(V,a_V)$ for $V$ an object in $\cC$ and $$a_V\colon A\otimes V\to V,$$ a morphism in $\cC$ satisfying 
$a_V(m\otimes \ide_V)=a_V(\ide_A \otimes  \; a_V)$ and $a_V(u\otimes \ide_V)=\ide_V$.
A morphism of $A$-modules $(V,a_V) \to (W, a_W)$ is a morphism $V \to W$ in $\cC$ that intertwines with $a_V$ and $a_W$. 
This way, we define the category $\lmod{A}(\cC)$ of left $A$-modules in $\cC$. Analogously, we define $\rmod{A}(\cC)$, the category of \emph{right $A$-modules in $\cC$}.

\smallskip

A {\it coalgebra} in $\cC$ is an object $C \in \cC$ equipped with two morphisms $\Delta\colon  C \to C \otimes C$ (comultiplication) and $\varepsilon\colon  C \to \one$ (counit) satisfying $(\Delta \otimes \ide_C) \Delta= (\ide_C \otimes\; \Delta)\Delta$ and $(\varepsilon \otimes \ide_C) \Delta= \ide_C = (\ide_C \otimes\; \varepsilon)\Delta$. Dual to above, we can define the category $\Coalg(\cC)$ of coalgebras and their morphisms in $\cC$, and given $C \in \Coalg(\cC)$ we can define categories, $\lcomod{C}(\cC)$ and $\rcomod{C}(\cC)$, of {\it left} and {\it right $C$-comodules in $\cC$}, respectively. For $V \in \lcomod{C}(\cC)$, the left $C$-coaction map is denoted by
$$\delta_V\colon  V \to C \otimes V.$$


\subsubsection{Bialgebras and Hopf algebras} \label{sec:Hopfalg}
In this part, let $\cC:=(\cC, \otimes, \one, c)$ be a braided  tensor category over $\Bbbk$. We define bialgebras and Hopf algebras in $\cC$ here, and more details can be found \cite{TV}*{Sections~6.1 and~6.2}.


\smallskip

A {\it bialgebra} in $\cC$ is a tuple $H:=(H,m,u,\Delta,\varepsilon)$ where $(H,m,u) \in {\sf Alg}(\cC)$ and $(H,\Delta,\varepsilon) \in {\sf Coalg}(\cC)$ so that $\Delta m=(m\otimes m)(\ide\otimes c \otimes \ide)(\Delta\otimes \Delta)$, $\Delta u = u \otimes u$, $\varepsilon m = \varepsilon \otimes \varepsilon$, and $\varepsilon u = \ide_{\one}$.  We denote by $\Bialg(\cC)$ the category of bialgebras in $\cC$, where morphisms in $\Bialg(\cC)$ are morphisms in $\cC$ that belong to $\Alg(\cC)$ and $\Coalg(\cC)$ simultaneously. 

\smallskip

A {\it Hopf algebra} is a tuple $H:=(H,m,u,\Delta,\varepsilon, S)$, where $(H,m,u,\Delta,\varepsilon) \in \Bialg(\cC)$ and $S\colon H \to H$ is a morphism in $\cC$ (called an {\it antipode}) so that 
$m(S \otimes \ide_H)\Delta = m(\ide_H \otimes S)\Delta = u \varepsilon$. We denote by $\HopfAlg(\cC)$ the category of Hopf algebras in $\cC$, where morphisms  are morphisms in $\Bialg(\cC)$. We assume that all Hopf algebras in this work have an invertible antipode, that is, there exists a morphism $S^{-1}\colon H \to H$ is a morphism in $\cC$ so that $S S^{-1} = \ide_H = S^{-1} S$.


\subsubsection{(Co)modules over Hopf algebras}
\label{sec:Hmod}
Now we discuss (co)modules over Hopf algebras $H$ in a braided tensor category $(\cC, \otimes, \one, c)$. 
We refer the reader to  \cite{EGNO}*{Sections~7.14,~7.15,~8.3} and \cite{Bes}*{Section~3} for more details.

\smallskip

If $V,W$ are left $H$-modules in $\cC$, then so is the tensor product $V \otimes W$, via the action \eqref{tensorproductaction} below:
\begin{equation}\label{tensorproductaction}a_{V\otimes W}:= (a_V\otimes a_W)(\ide_H\otimes c_{H,V}\otimes \ide_W)(\Delta_{H} \otimes \ide_{V\otimes W}).
\end{equation}
This makes the category $\lmod{H}(\cC)$  a monoidal category, with unit object $(\one=\Bbbk, a_{\one} = \varepsilon_{H} \otimes \ide_{\one})$.
Assume that $\cC$ is rigid, and take $(V,a_V) \in \lmod{H}(\cC)$. Then its left dual 
 $(V^*,a_{V^*}) \in \lmod{H}(\cC)$ is defined using $S_H$, and its right dual $({}^*V,a_{{}^*V}) \in \lmod{H}(\cC)$ is defined using $S^{-1}_H$. 
 It follows that $\lmod{H}(\cC)$ is a (finite) tensor category provided $\cC$ is a (finite) braided tensor category.
 
\smallskip

 For one supply of braided tensor categories, take the category $\lmod{H}(\Vect)$ for  $H$ a finite-dimensional $\Bbbk$-Hopf algebra. We say that $H$ is {\it quasi-triangular} if it comes equipped with an invertible element $R = R^{(1)} \otimes R^{(2)} \in H \otimes H$ satisfying
$$(\Delta \otimes \ide)(R) = R_{13} R_{23}, \quad
(\ide \otimes \Delta)(R) = R_{13}R_{12}, \quad
\Delta^{\text{op}}(h) = R \Delta(h) R^{-1}, \quad \text{ for } h \in H,$$
where $\Delta^{\text{op}}$ is the opposite coproduct.
It follows that $\lmod{H}(\Vect)$ is a braided tensor category if and only if the finite-dimensional Hopf algebra $H$ is quasi-triangular; here, the braiding is given by 
$$c_{V,W}(v\otimes w)= a_W(R^{(2)} \otimes w)\otimes a_V(R^{(1)}\otimes v),$$
for $(V,a_V), (W, a_W) \in \lmod{H}(\Vect)$. We say that $H$ is {\it coquasi-triangular} if it comes equipped with a convolution-invertible bilinear form $r\colon  H \otimes H \to \Bbbk$ satisfying
\begin{gather*}
r(h,k\ell) = r(h_{(1)},\ell)r(h_{(2)},k), \quad r(\ell h,k) = r(\ell,k_{(1)})r(h,k_{(2)}), \\ r(h_{(1)},\ell_{(1)})h_{(2)}\ell_{(2)} = \ell_{(1)}h_{(1)} r(h_{(2)},\ell_{(2)}),
\end{gather*}
$h,k,\ell \in H$ (see e.g. \cite{Majid}*{Definition 2.2.1}).
It follows that $\lcomod{H}(\Vect)$ is a braided tensor category if and only if the finite-dimensional Hopf algebra $H$ is coquasi-triangular; here, the braiding is given by 
$$c_{V,W}(v\otimes w)  = (r \otimes \ide_W \otimes \ide_V)(\ide_H \otimes \tau \otimes \ide_V)(\delta_W \otimes \delta_V)(w \otimes v)$$
for  $(V,\delta_V), (W, \delta_W) \in \lcomod{H}(\Vect)$ and $\tau(a \otimes b) = b \otimes a$.

\smallskip

For another supply of braided tensor categories, take a Hopf algebra $H$ in $\cC$, and consider the category of {\it $H$-Yetter--Drinfeld modules in $\cC$}, denoted by $\lYD{H}(\cC)$, which consists of objects $(V,a_V,\delta_V)$, where $(V,a_V) \in \lmod{H}(\cC)$ with left $H$-coaction  in $\cC$ denoted by $\delta_V\colon  V \to H \otimes V$, subject to compatibility condition:
\[
\begin{split} 
\label{eq:HYDB}
&(m_H \otimes a_V)(\ide_H \otimes c_{H,H} \otimes \ide_V)(\Delta_H \otimes \delta_V)\\
&\quad = (m_H \otimes \ide_V)(\ide_H \otimes c_{V,H})(\delta_V \otimes \ide_H) (a_V \otimes \ide_H) (\ide_H \otimes c_{H,V})(\Delta_H \otimes \ide_V).
\end{split}
\]
A morphism $f\colon (V,a_V,\delta_V)\to (W,a_W,\delta_W)$ in $\lYD{H}(\cC)$ is given by a morphism $f\colon V\to W$ in $\cC$ that belongs to $\lmod{H}(\cC)$ and $\lcomod{H}(\cC)$. Given two objects $(V,a_V,\delta_V)$ and $(W,a_W,\delta_W)$ in $\lYD{H}(\cC)$, their tensor product  is given by $(V\otimes W, a_{V\otimes W}, \delta_{V\otimes W})$, where
 $a_{V\otimes W}$  as in \eqref{tensorproductaction} and 
\begin{align*}
\delta_{V\otimes W}=(m_H\otimes \ide_{V\otimes W})(\ide_H\otimes c_{H,V}\otimes \ide_W)(\delta_V\otimes \delta_W)
.
\end{align*}
The category $\lYD{H}(\cC)$ is braided with braiding  given by
$c^{\sf YD}_{V,W} = (a_W \otimes \ide_V)(\ide_H \otimes c^{\cC}_{V,W})(\delta_V \otimes \ide_W).$
Further, when $\cC = \Vect$ and $\dim_\Bbbk  H < \infty$, we get that $\lYD{H}(\Vect)$  is equivalent to the braided tensor category of modules over the Drinfeld double, $\Drin(H)$, see e.g. \cite{Majid}*{Theorem 7.1.2} and cf. Example \ref{expl:br-Drin} below with $K=\Bbbk$.

\subsection{Ribbon tensor categories}
\label{sec:ribbon}
In this section we assume that $\cC := (\cC, \otimes, \one, c)$ is a braided tensor category, and we refer the reader to \cite{BK}*{Chapter~2}, \cite{EGNO}*{Sections~8.9--8.11},  \cite{TV}*{Section~3.3}, and \cite{Rad}*{Chapter~12} for details of the discussion below.

\smallskip

A braided tensor category $(\cC, \otimes, \one, c)$  is {\it ribbon} (or \emph{tortile}) if  it is equipped with a natural isomorphism $\theta_X\colon X \overset{\sim}{\to} X$  (a {\it twist}) satisfying $\theta_{X \otimes Y} = (\theta_X \otimes \theta_Y)  \circ c_{Y,X} \circ c_{X,Y}$ and $(\theta_X)^* = \theta_{X^*}$ for all $X,Y \in \cC$. A \emph{functor (or, equivalence) of ribbon categories} is a functor (respectively, equivalence) $F\colon \cC\to \cD$ of braided  tensor  categories such that $F(\theta_X^{\cC})=\theta^{\cD}_{F(X)}$, for any $X\in \cC$, cf. \cite{Shum}*{Section~1}.

\smallskip

In a ribbon category $(\cC,\otimes,\one,c,\theta)$, consider the {\it Drinfeld isomorphism}:
\begin{align}\label{Driniso}
\phi_X=(\ide_{X^{**}}\otimes \ev_X)(c_{X^*,X^{**}}\otimes \ide_{X})(\coev_{X^*}\otimes\ide_X)\colon X \overset{\sim}{\to} X^{**}.
\end{align}
Then,
\begin{align}\label{eq:ribbonpivotal} 
j_X:=\phi_X\theta_X\colon X\isomorph X^{**}
\end{align}
defines a pivotal structure on $\cC$.

\smallskip

For a supply of ribbon categories,  consider the category $\lmod{H}(\Vect)$ for  $H = (H,R)$ a finite-dimensional quasi-triangular $\Bbbk$-Hopf algebra. We say that $H$ is a \emph{ribbon Hopf algebra} if there exists a central invertible element $v\in H$ satisfying
\begin{align}\label{ribbonelement}
\Delta(v)=(R_{21}R)^{-1}(v\otimes v), \qquad \varepsilon(v)=1, \qquad S(v)=v.
\end{align}
This definition is equivalent to the one given in \cite{RT90}*{Section 3.3}, \cite{Rad}*{Definition 12.3.5}. 
It follows that $\lmod{H}(\Vect)$ is a ribbon category if and only if $H$ is a ribbon Hopf algebra \cite{Majid}*{Corollary~9.3.4}. 
In this case, the ribbon twist is given by the action of $v^{-1}$.



\smallskip

The following lemma will be of use.

 \begin{lemma} \label{lem:ribbon-subcat}
 Take $\cD$ a braided full tensor subcategory of a braided tensor category~$\cC$. If $\cC$ is ribbon, then so is $\cD$.  \qed
 \end{lemma}
 
This result is obtained by restricting the ribbon structure from $\cC$ to $\cD$.  Moreover, the ribbon structure of the monoidal center $\cZ(\cC)$ will be discussed later in Section~\ref{sec:centerribbon}.


\subsection{Modular tensor categories}
\label{sec:modular}
In the section, we discuss a notion of a modular tensor category for the non-semisimple setting. This is based on work of Kerler--Lyubashenko \cite{KL} and recent work of Shimizu \cite{Shi1}. To proceed, we adopt the definition of non-degeneracy below, which extends the notion of non-degeneracy in the semisimple setting; see \cite{EGNO}*{Definition~8.19.2 and Theorem~8.20.7}.

\begin{definition}[{\cite{Shi1}*{Theorem~1}}] 
\label{def:nondeg} 
We call a braided finite tensor category $(\cC, \otimes, \one, c)$ {\it non-degenerate}  if its M\"uger center $\cC'$ is equal to $\Vect$. 
\end{definition}

Next we discuss a characterization of non-degeneracy. 
Let $(\cC, \; c_{X,Y}\colon X \otimes Y \overset{\sim}{\to} Y \otimes X)$ be a braided tensor category, and take the braided tensor category:
$$\overline{\cC} := (\cC,\; c^{-1}_{Y,X}\colon X \otimes Y \overset{\sim}{\to} Y \otimes X).$$ The assignments $\cC \to \cZ(\cC)$, $X \mapsto (X, c_{X,-})$,  and $\overline{\cC} \to \cZ(\cC)$, $X \mapsto (X, c^{-1}_{-,X})$, extend to a braided tensor functor $\cC \boxtimes \overline{\cC} \to \cZ(\cC)$. If this functor yields an equivalence between the braided tensor categories $\cC \boxtimes \overline{\cC}$ and $\cZ(\cC)$, then we say that $(\cC, \otimes, \one, c)$  is {\it factorizable}.
A braided finite tensor category is non-degenerate if and only if it is factorizable \cite{Shi1}*{Theorem~4.2}; note that this article also provides a third equivalent characterization of non-degeneracy in terms of a non-degenerate form on the coend.

\smallskip

Moreover, the following type of tensor  categories are of primary interest in this work.

\begin{definition}[{\cite{KL}*{Definition~5.2.7}, \cite{Shi1}*{Section 1}}] \label{def:modular}
A braided finite tensor category is called {\it modular} if it is non-degenerate and ribbon.
\end{definition}

Now consider the braided finite tensor category $\lmod{H}(\Vect)$ for $H$ a finite-dimensional, quasi-triangular Hopf algebra over $\Bbbk$. We get that $\lmod{H}(\Vect)$ is modular precisely when $H$ is ribbon and factorizable \cite{EGNO}*{Proposition~8.11.2 and Example~8.6.4}.

\begin{remark} \label{rem:modular-subcat}
By Lemma~\ref{lem:ribbon-subcat}, we obtain that a topologizing non-degenerate braided tensor subcategory of a modular category is also modular.
\end{remark}

\begin{remark} \label{rem:Deligne-modular}
It is straight-forward to show that if $\cC$ and $\cD$ are modular, then so is $\cC \boxtimes \cD$ via the monoidal structure \eqref{eq:Deligne-monoidal}, the braiding \eqref{eq:Deligne-braiding}, and with ribbon structure $\theta^{\cC \boxtimes \cD} := \theta^\cC \boxtimes \theta^\cD$.
\end{remark}


\section{Non-semisimple spherical categories and ribbon structures on the center}\label{sec:spherical}
Let $\cC$ be a finite tensor category over an algebraically closed field $\Bbbk$. The purpose of this section is to review sufficient conditions for the monoidal center $\cZ(\cC)$ to be a modular tensor category. First, we recall the notion of a distinguished invertible object and the Radford isomorphism of $\cC$ in Section~\ref{sec:Ddelta}. This allows us to recall, in Section~\ref{sec:centerribbon}, Shimizu's necessary and sufficient conditions for $\cZ(\cC)$ to be ribbon, generalizing a result of Radford--Kauffman  in the case when $\cC =  \lmod{H}(\Vect)$ for $H$ a finite-dimensional Hopf algebra. In Section~\ref{sec:nonss-sph}, we recall the concept of a spherical category introduced in the work of Douglas--Schommer-Pries--Snyder \cite{DSS}, expanding the semisimple notion in \cite{BW} to the non-semisimple setting; it is then applied to describe when $\cZ(\cC)$ is modular.

\subsection{The distinguished invertible object \texorpdfstring{$D$}{D} and the Radford isomorphism \texorpdfstring{$\xi_{D}$}{xiD}}
\label{sec:Ddelta}
For details, see \cite{EGNO}*{Sections~7.18--7.19 and~8.10}, \cite{Rad}*{Section~10.5}, and references within.

\smallskip

Consider $\cC$ as a $\cC \boxtimes \cC^{\otimes \text{op}}$-module category.
In this case, the {\it canonical algebra} is defined as $A_{\text{can}}:=\iHom(\one,\one) \in {\sf Alg}(\cC \boxtimes \cC^{\otimes \text{op}}).$ For example, if  $H$ is a Hopf algebra over $\Bbbk$, then the canonical algebra in $\cC:=\lmod{H}(\Vect)$ is $H^* \in {\sf Alg}(\cC \boxtimes \cC^{\otimes \text{op}})$ viewed as an $H$-bimodule over $\Bbbk$ with left and right $H$-actions given by translation.
The category ${\sf HopfBimod}(\cC):=\rmod{A_{\text{can}}}(\cC \boxtimes \cC^{\otimes \text{op}})$ of right $A_{\text{can}}$-modules in $\cC \boxtimes \cC^{\otimes \text{op}}$  is called the 
{\it category of Hopf bimodules in $\cC$}. 
Both $A_{\text{can}}$ and its dual object $A_{\text{can}}^*$ belong to ${\sf HopfBimod}(\cC)$. Moreover,  ${\sf HopfBimod}(\cC)$ is a tensor subcategory of $(\cC \boxtimes \cC^{\otimes \text{op}}, \; \odot,\; A_{\text{can}})$ where 
$$\odot:= \rho(\ide_{\cC} \times F) \boxtimes \ide_{\cC^{\otimes \text{op}}}\colon (\cC \boxtimes \cC^{\otimes \text{op}}) \times (\cC \boxtimes \cC^{\otimes \text{op}})  \cong (\cC \times (\cC^{\otimes \text{op}} \boxtimes \cC)) \boxtimes \cC^{\otimes \text{op}} \longrightarrow \cC \boxtimes \cC^{\otimes \text{op}},$$
for $F\colon  \cC^{\otimes \text{op}} \boxtimes \cC \to \Vect, (X,Y) \mapsto \Hom_\cC({}^* Y,X)$ and $\rho$ is the natural action of $\Vect$ on $\cC$.
Continuing the example above, for $\cC:=\lmod{H}(\Vect)$ with $H \in \HopfAlg(\Vect)$, and $A_{\text{can}} = H^*$, we get that ${\sf HopfBimod}(\cC)$ is the usual category of Hopf bimodules over $H$. Moreover, for $M,N \in {\sf HopfBimod}(\cC)$, we obtain that $M \odot N = (M^* \otimes_H N^*)^*$, for $()^*$ denoting the $\Bbbk$-linear dual here. 

\smallskip

By \cite{ENO}*{Theorem~3.3} and \cite{DSS}*{Theorem~3.3.4}, there exists a invertible object $D \in \cC$  so that 
$(D \boxtimes \one) \odot A_{\text{can}} \cong A_{\text{can}}^*$ as objects in ${\sf HopfBimod}(\cC)$. This isomorphism is unique up to a scalar, and $D$ is indeed an invertible object of $\cC$. We call $D$ the {\it distinguished invertible object} of $\cC$. 

\smallskip

We also get  a canonical natural tensor isomorphism $$\xi_{D}(X)\colon D \otimes X \overset{\sim}{\to} X^{4*}\otimes D,$$
defined as follows. Let $F,G\colon  \cC \to \cC$ be two tensor functors, and  consider the category $\cZ(F,G)$ with objects:
$$\Ob(\cZ(F,G)) = \{(V,\sigma_V) ~|~ V \in \cC, \; \sigma_V(-)\colon  V \otimes F(-) \overset{\sim}{\to} G(-) \otimes V ~\text{a $\otimes$-compatible natural isom.}  \},$$
where the compatibility conditions are  \cite{Shi2}*{(3.1),(3.2)}.
Two objects $(V,\sigma_V)$ and $(V',\sigma_{V'})$ of $\cZ(F,G)$ are {\it equivalent} if there is an isomorphism $f\colon  V \overset{\sim}{\to} V'$ so that $\sigma_{V'}(X)(f \otimes \ide_{F(X)}) = (\ide_{G(X)} \otimes f)\sigma_V(X)$ for all $X \in \cC$. This category is not always monoidal, but it is always a finite abelian category \cite{Shi2}*{Theorem~3.4}. We also have that $\cZ(\ide_\cC,\ide_\cC)$ is the monoidal center $\cZ(\cC)$. By \cite{Shi2}*{Lemma~3.3,~(4.3)}, we get equivalences $$\cZ(\ide_\cC,(-)^{4*}_\cC)
\overset{\sim}{\longrightarrow}
\lmod{A_{\text{can}}^{**}}(\cC)
\overset{\sim}{\longrightarrow}
A_{\text{can}}^{**}\text{-}{\sf HopfBimod}(\cC).$$ 
The first equivalence is an isomorphism, 
and the second equivalence is induced by 
$\cC \overset{\sim}{\rightarrow}
{\sf HopfBimod}(\cC)$ given by $Y \mapsto (Y \boxtimes \one) \odot A_{\text{can}}$.
Now the object $A_{\text{can}}^{*}$ in $A_{\text{can}}^{**}$-{\sf HopfBimod}($\cC$) corresponds to pair $(D,\xi_{D})$ in $\cZ(\ide_\cC,(-)^{4*}_\cC)$. Here, $\xi_{D}$ is called the {\it Radford isomorphism} of $\cC$.

\smallskip

Now consider the case $\cC=\lmod{H}(\Vect)$ for a Hopf algebra $H$ over $\Bbbk$, and consider the {\it distinguished grouplike elements} of $H$ and $H^*$ defined as follows (see \cite{KR93}*{Section~1} or \cite{Rad}*{Section 10.5}).
 In this case, $D$ is a one-dimensional module, and so the action is given through an invertible character $\alpha_{H} \in H^*$, i.e. $h\cdot d=\alpha_{H}^{-1}(h)d$ for any $d\in D$. By virtue of $D$ being the distinguished invertible object in $\cC$,
 \begin{equation} \label{eq:alpha_H}
     \text{$\alpha_{H}$ $\in G(H^*)$ is uniquely characterized by 
$\alpha_{H}(h)\Lambda=\Lambda h,$ for all $h\in H,$}
 \end{equation}
 for a choice of non-zero left integral $\Lambda$ for $H$.
 The Radford isomorphism is now given by the action of an element
 \begin{equation} \label{eq:g_H}
   \text{$g_{H}$ $\in G(H)$, which is uniquely characterized by
$p\lambda = \ev(p\otimes g_{H})\lambda,$  for all $p\in H^*,$}
 \end{equation} 
where $\lambda$ is a non-zero right integral of $H^*$. 
Explicitly, if $D=\Bbbk v$, then for any $X \in \cC$ and $x \in X$ we get
$\xi_{D}(X)(v\otimes x)= g_{H}\cdot x\otimes v.$
Now, Radford's $S^4$-formula \cite{Rad76}, 
$$S^4(h)=\alpha_{H}^{-1}(h_{(1)})\;g_{H}\;h_{(2)}\;g_{H}^{-1}\;\alpha_{H}(h_{(3)}),$$ implies that
$(D,\xi_{D})$ defines an object in $\cZ(\ide_\cC,(-)^{4*})$.

\smallskip

 Recall that a finite tensor category $\cC$ is \emph{unimodular} if $D=\one$ \cite{EGNO}*{Section~6.5}.
When $\cC$ is a factorizable finite tensor category, then $\cC$ is unimodular  \cite{EGNO}*{Proposition~8.10.10}.

\subsection{Ribbon structures on the center}\label{sec:centerribbon}

In this section, we recall results of \cite{Shi2} and \cite{KR93} on the existence of ribbon structures on the center $\cZ(\cC)$ of a finite tensor category, using the pair $(D, \xi_{D})$ defined in the previous section.

\begin{definition}[$\mathsf{Sqrt}_\cC(D,\xi_{D})$] \label{def:sqrtset}
Let $\cC$ be a finite tensor category and recall $(D,\xi_{D})$  from Section~\ref{sec:Ddelta}. We define $\mathsf{Sqrt}_\cC(D,\xi_{D})$ to be the set of equivalence classes of invertible objects $(V,\sigma_V)$ in $\cZ(\ide_\cC,(-)^{**})$ such that there exists an isomorphism $\nu\colon V^{**}\otimes V\isomorph D$ such that the following diagram commutes:
\begin{align}\label{diag:Shimizu}
\vcenter{\hbox{
\xymatrix{
V^{**} \otimes V\otimes X\ar[rrr]^{\ide_{V^{**}}\otimes \sigma_V(X)} \ar[d]_{\nu\otimes \ide_X}&&& V^{**}\otimes X^{**}\otimes V \ar[rrr]^{\sigma_V(X)^{**}\otimes \ide_V} &&&X^{4*}\otimes V^{**}\otimes V\ar[d]^{\ide_{X^{4*}}\otimes \nu}\\
D \otimes X\ar[rrrrrr]^{\xi_D(X)} &&&&&& X^{4*}\otimes D.
}}}
\end{align}
\end{definition}

\smallskip

\begin{theorem}[\cite{Shi2}*{Sections 5.2, 5.4, 5.5}]
\label{thm:Shimizuribbon}
The set of ribbon structures on $\cZ(\cC)$ is in bijection with the set $\mathsf{Sqrt}_\cC(D,\xi_{D})$.
 In particular, if we take $(V,\; \sigma_V) \in \mathsf{Sqrt}_\cC(D,\xi_{D})$, then the corresponding ribbon structure on $\cZ(\cC)$ is given by $\theta = \phi^{-1} j$, where $\phi$ is the  Drinfeld isomorphism of $\cZ(\cC)$ from~\eqref{Driniso} and  $j$ is a pivotal structure on $\cZ(\cC)$ given by 
 $$j_X := (\ide_{X^{**}} \otimes \coev_V^{-1})(\sigma_V(X) \otimes \ide_{V^*})(c_{X,V} \otimes \ide_{V^*})(\ide_X \otimes \coev_V)\colon X \overset{\sim}{\to} X^{**}, $$
 for $X \in \cC$.
\qed
\end{theorem}

In the case when $\cC=\lmod{H}(\Vect)$, Shimizu's theorem specializes to the following classical result of Kauffman--Radford.

\begin{theorem}[\cite{KR93}*{Theorem~3}]\label{thm:KR}
Let $g_H\in H$ and $\alpha_H\in H^*$ be the distinguished grouplike elements defined in \eqref{eq:alpha_H}, \eqref{eq:g_H}. 
Then  there is a bijection between the sets 
\begin{gather}
\left\{(\zeta,a)\in G(H^*)\times G(H)\;\big|\; \zeta^2=\alpha_H, \; \; a^2=g_H, \;\; \text{satisfying \eqref{eq:S2}} \right\},\qquad \text{ where }\nonumber \\
\label{eq:S2}
   S^2(h)=\zeta^{-1}(h_{(1)})ah_{(2)}a^{-1}\zeta(h_{(3)}), \qquad \text{ for all }h\in H,
\end{gather}
and the set of ribbon elements of the Drinfeld double, $\Drin(H)$, cf. \eqref{ribbonelement}.  

The bijection is given by sending a pair $(\zeta,a)$ to $S(R^{(2)})R^{(1)}(\zeta^{-1} \otimes a^{-1})$, for $R$ and $S$ the R-matrix and antipode of $\Drin(H)$, respectively.  \qed
\end{theorem}

The precise connection between the above results of Kauffman--Radford and Shimizu is given by the following proposition.

\begin{proposition}\label{prop:KRvsShi}
 Let $H$ be a finite-dimensional Hopf algebra  and $\cC=\lmod{H}$. Then there is a bijection between the set of pairs $(\zeta,a)$ of Theorem \ref{thm:KR} and the set $\mathsf{Sqrt}_\cC(D,\xi_{D})$ of Definition \ref{def:sqrtset}.
\end{proposition}

\begin{proof}  
Given a pair $(\zeta,a)$ as in Theorem \ref{thm:KR}, we define $(V,\sigma_V)\in \cZ(\ide,(-)^{**})$ of $\mathsf{Sqrt}_\cC(D,\xi_{D})$ as follows. First, $V$ is the one-dimensional $H$-module with action $h\cdot v=\zeta^{-1}(h)v$ for any $v\in V,h\in H$. Second, the isomorphism $\sigma_V(X)\colon V\otimes X\to X^{**}\otimes V$ is defined by 
$\sigma_X(v\otimes x)=ax\otimes v$ for all $v\in V, x\in X.$
This isomorphism defines an element in $\mathsf{Sqrt}_\cC(D,\xi_{D})$ provided that $(\zeta,a)$ satisfy the conditions of Theorem \ref{thm:KR}. In particular, \eqref{eq:S2} implies that $\sigma_V(X)$ is a morphism of $H$-modules.

Conversely, assume given a pair $(V, \sigma_V) \in \mathsf{Sqrt}_\cC(D,\xi_{D})$. Then $V$ is an invertible $H$-module and thus is $1$-dimensional. Fix a generator $v\in V$. Then we obtain $\zeta$ such that $h\cdot v=\zeta^{-1}(h)v$. The isomorphism $\nu\colon V^{**}\otimes V\to D$ of $H$-modules implies that $\zeta^{-2}=\alpha_H^{-1}$ and hence $\zeta^2=\alpha_H$. We obtain an element $a \in G(H)$ so that $a^2 = g_H$, satisfying, \eqref{eq:S2} as follows. Recall that there is an isomorphism of Hopf algebras $H\cong \End(F)$, where $F\colon \cC\to \Vect$ is the forgetful functor, for details, see e.g. \cite{EGNO}*{Sections 5.2--5.3}. Here, $h \in H$ gets sent to $\{h_X\colon  F(X) \to F(X), x \mapsto h \cdot x\}_{X \in {\sf Ob}(\cC)}$.
Further, for the 1-dimensional $H$-module $V$ fixed above, there are isomorphisms of $\Bbbk$-vector spaces
$f'_X\colon  F(V)\otimes F(X)\isomorph F(X),  v\otimes x\mapsto x$ and $f''_X\colon  F(X) \isomorph F(X) \otimes F(V), x \mapsto x \otimes v,$ natural in $X$.
So by identifying the $\Bbbk$-vector spaces $F(X)$ and $F(X^{**})$, we obtain that the natural isomorphism
$\sigma_V(X)\colon F(V)\otimes F(X)\isomorph F(X)\otimes F(V)$ is of the form $f''_X \circ a_X\circ f'_X$ and must be given by $v \otimes x \mapsto (a \cdot x) \otimes v$ for some $a \in H$. 
The assumption that $\sigma_V$ defines a natural isomorphism $V\otimes (-)\isomorph (-)^{**}\otimes V$ of $H$-modules implies condition \eqref{eq:S2}. The diagram in \eqref{diag:Shimizu} implies that $a^2=g_H$.
\end{proof}

\subsection{Non-semisimple spherical categories}
\label{sec:nonss-sph}

Using the distinguished invertible object $D$ defined in Section~\ref{sec:Ddelta}, we obtain a notion of sphericality for non-semisimple finite tensor categories.

 \begin{definition} [\cite{DSS}*{Definition 3.5.2}] \label{def:spherical}
A pivotal finite tensor category $(\cC, \otimes, \one, j)$ is {\it spherical} if  there is an isomorphism $\nu\colon \one \overset{\sim}{\to} D$ so that the following diagram commutes
\begin{align}\label{eq:spherical}
\vcenter{\hbox{
\xymatrix{
X\ar[rr]^{j_{X}^{**}\; j_{X}}\ar[d]_{\nu\otimes \ide_X}&&X^{4*}\ar[d]^{\ide_{X^{4*}}\otimes \nu}\\
D \otimes X\ar[rr]^{\xi_D(X)} && X^{4*}\otimes D.
}}}
\end{align}
\end{definition}

In fact, if $\cC$ is semisimple; then $\cC$ is spherical precisely when $\cC$ is trace-spherical; see \cite{DSS}*{Proposition 3.5.4}.

\smallskip

\begin{remark} \label{rem:sph-rib}
On the one hand, a spherical category in the sense above gives a special case of a tensor category $\cC$ satisfying the assumption that $\mathsf{Sqrt}_\cC(D,\xi_{D})$ from Definition~\ref{def:sqrtset}  is non-empty; namely, $(\one,j)\in \mathsf{Sqrt}_\cC(D,\xi_{D})$. 
On the other hand, Example \ref{expl:taft} later illustrates that there are categories $\cC$ satisfying $\mathsf{Sqrt}_\cC(D,\xi_{D})\neq \varnothing$ that do not have a spherical structure.
\end{remark}

\begin{proposition}[$\mathsf{SPiv}(H)$] \label{prop:Hspherical} Take $\cC=\lmod{H}(\Vect)$ and recall \eqref{eq:alpha_H}, \eqref{eq:g_H}. Then $\cC$ is spherical in the sense of Definition \ref{def:spherical} if and only if  $\alpha_H=\varepsilon$ and 
\begin{align*}
    \mathsf{SPiv}(H):= \big\{ a\in G(H)\mid a^2=g_H,\; S^2(h)=aha^{-1}, \text{ for all $h\in H$}\big\}\neq \varnothing.
\end{align*}
In this case, there is a bijection between pivotal structures $j$ on $\cC$ such that $(\cC,j)$ is spherical and the set $\mathsf{SPiv}(H)$.
\end{proposition}

\begin{proof}
First assume that $\cC=\lmod{H}(\Vect)$ is spherical. Then, by definition, $D\cong \one$, which implies that $\alpha_H=\varepsilon$. By Remark~\ref{rem:sph-rib}, $(\one, j)\in \mathsf{Sqrt}_\cC(D,\xi_{D})$. Using Proposition \ref{prop:KRvsShi}, this element of $\mathsf{Sqrt}_\cC(D,\xi_{D})$ corresponds to a pair $(\zeta,a)$ satisfying \eqref{eq:S2}, with $\zeta=\varepsilon$, such that $a^2=g_H$. Thus, $a\in \mathsf{SPiv}(H)$. 
From Proposition \ref{prop:KRvsShi} it further follows that there is a bijection between the subset of 
$\mathsf{Sqrt}_\cC(D,\xi_{D})$ of pairs $(V,\sigma_V)$ such that $V\cong\one$ and pairs $(\varepsilon, a)$ satisfying the conditions of Theorem~\ref{thm:KR}  (i.e $a^2=g_H$ and \eqref{eq:S2}, or equivalently, $a\in \mathsf{SPiv}(H)$).

Conversely, assume $\alpha_H=\varepsilon$ and $a\in  \mathsf{SPiv}(H)\neq \varnothing$. Again under the bijection of Proposition \ref{prop:KRvsShi}, the pair $(\varepsilon,a)$ corresponds an element $(\one,\sigma_\one)\in \mathsf{Sqrt}_\cC(D,\xi_{D})$. 
In particular, $(\one,\sigma_\one)$ is an element of $\cZ(\ide_{\cC},(-)^{**})$ which implies, using the convention $\one\otimes X=X=X\otimes \one$, that $j:=\sigma_\one$ is a pivotal structure for $\cC$, which, by construction, satisfies \eqref{eq:spherical}. Thus, $\cC$ is spherical.
\end{proof}

Examples of spherical categories obtained from Nichols algebras can be found later in the text, see Remark \ref{rmk:UBqmod}(3) and Example \ref{expl-new}.

\smallskip

Next, we show that a source of non-semisimple spherical categories is given by unimodular ribbon categories, cf. \cite{EGNO}*{Proposition~8.10.12} in the semisimple case. Recall that in the semisimple case, every finite tensor category is unimodular \cite{EGNO}*{Remark~6.5.9}.

\begin{proposition}\label{prop:ribbonspherical}
 Any unimodular finite ribbon category is spherical in the sense of Definition~\ref{def:spherical}. 
\end{proposition}

\begin{proof}
Assume $\cC$ is a ribbon category with braiding $c$ and twist $\theta$. Then $\cC$ is a pivotal category via the pivotal structure $j$ of \eqref{eq:ribbonpivotal}.  
Consider the following computation:
\begin{align*}
\phi_{X^*}^*\; j_X^{**}\; j_X &= \phi_{X^*}^*\; j_{X^{**}}\; j_X\\
&= \phi_{X^*}^* \; \phi_{X^{**}} \; \theta_{X^{**}}\; \phi_X \; \theta_X\\
&= \phi_{X^*}^* \; \phi_{X^{**}} \;
(\ide_{X^{**}} \otimes \ev_X)(c_{X^{*}, X^{**}}(\theta_{X^*} \otimes \theta_{X^{**}})\coev_{X^*} \otimes \ide_X)
\\
&= \phi_{X^*}^* \; \phi_{X^{**}} \;
(\ide_{X^{**}} \otimes \ev_X)(c^{-1}_{X^{**}, X^{*}}(\theta_{X^* \otimes X^{**}})\coev_{X^*} \otimes \ide_X)
\\
&= \phi_{X^*}^* \; \phi_{X^{**}} \;
(\ide_{X^{**}} \otimes \ev_X)(c^{-1}_{X^{**}, X^{*}} \otimes \ide_X)
\\
&=\phi_X.
\end{align*}
 Here, the first equation uses that $j_{X^*}=j_X^*$. The second equation uses \eqref{eq:ribbonpivotal} and that $\theta_{X^*}=\theta_X^*$. The third equality uses \eqref{Driniso} and the naturality of $\theta$. The fourth equation follows from $\theta_{X\otimes Y} = (\theta_X \otimes \theta_Y) \; c_{X,Y}\; c_{Y,X} =   c_{Y,X}\; c_{X,Y} \; (\theta_X \otimes \theta_Y)$. The next equation holds by the naturality of $\theta$ and the fact that $\theta_\one=\ide_\one$. The last equality then follows a sequence of arguments using the naturality of the braiding and rigidity axioms.
 So, $j_X^{**}j_X = (\phi_{X^*}^*)^{-1} \phi_X$.

 Now assume that $\cC$ is unimodular so that $D=\one$.
Using \cite{Shi2}*{Theorem~A.6},  cf. \cite{EGNO}*{Theorem 8.10.7}, we find that 
$$\xi_D(X)=(\phi_{X^*}^*)^{-1} \phi_X.$$
Thus, using $\eta=\ide\colon \one \to D$, this shows that $j$ satisfies the conditions of Definition \ref{def:spherical}. Thus, $\cC$ is spherical. 
\end{proof}

Finally, recent results of Shimizu provide, in the non-semisimple framework, a sufficient condition for the monoidal center $\cZ(\cC)$ to be modular.

\begin{theorem}[{\cite{Shi2}*{Theorem~5.10}}]
\label{thm:Shimizucenter}
If $\cC$  is a tensor category with $\mathsf{Sqrt}_\cC(D,\xi_{D}) \neq \varnothing$, then its center $\cZ(\cC)$ is modular in the sense of Definition~\ref{def:modular}. \qed
\end{theorem}

By Remark~\ref{rem:sph-rib}, one obtains the following consequence generalizing M\"uger's result \cite{Mue2}*{Theorem 2}, see also \cite{TV}*{Theorem~9.11}, to the non-semisimple case.

\begin{corollary}[{\cite{Shi2}*{Theorem~5.11}}]
\label{cor:Shimizucenter}
If $\cC$ is spherical in the sense of Definition~\ref{def:spherical}, then its center $\cZ(\cC)$ is modular in the sense of Definition~\ref{def:modular}. \qed
\end{corollary}


\section{Modularity of M\"uger centralizers and relative monoidal centers}\label{sec:ZBC-modular}

This section contains the main categorical results of this paper. We first use the double centralizer theorem of \cite{Shi1} to prove that M\"uger centralizers of non-degenerate topologizing subcategories in a modular category are again modular [Section~\ref{sec:muger}]. Then we recall the construction of the relative center $\cZ_\cB(\cC)$ [Section~\ref{sec:Bcenter}] and conclude, as a main application, that it produces modular categories under conditions on $\cC$ identified in \cite{Shi2} and assuming that $\cB$ is non-degenerate [Section~\ref{sec:Bcentralmodular}]. We also produce an analogue of M\"uger's decomposition theorem of modular categories in the non-semisimple setting [Section~\ref{sec:decomposition}].

\subsection{Modularity of M\"uger centralizers}
\label{sec:muger}
In the following we apply the double centralizer theorem of Shimizu \cite{Shi1}*{Theorem~4.9} (which is a non-semisimple version of \cite{Mue}*{Theorem 3.2(i)}) to obtain a generalization of the result of M\"uger \cite{Mue}*{Corollary~3.5} that centralizers of non-degenerate subcategories in modular categories are again modular in the non-semisimple setting. For this, we require the following notion of centralizer.

\smallskip

Let $\cS$ be a subset of objects of a braided category $\cC$, the \emph{M\"uger centralizer}  $\sfC_{\cC}(\cS)$ \cite{Mue}*{Definition~2.6} of $\cS$ in $\cC$ is defined as the full subcategory of $\cC$ with objects
\begin{equation} \label{eq:centralizer}
    \Ob(\sfC_{\cC}(\cS)) := \{ X \in \cC ~|~ c_{Y,X} \; c_{X,Y}=\ide_{X\otimes Y} \text{ for all } Y \in \cS\}.
\end{equation}
Note that $\sfC_{\cC}(\cS)$  is a topologizing monoidal subcategory of $\cC$ and, thus, braided. For a single object $X$ in $\cC$, we denote $\sfC_\cC(\Set{X})=\sfC_\cC(X)$. If $\cC$ is rigid, then $\sfC_\cC(\cS)$ is a rigid monoidal subcategory of $\cC$ \cite{Mue}*{Lemma 2.8}. The following result is straight-forward.

\begin{lemma}\label{lem:cent-finite}
 If $\cC$ is a (finite) tensor category and $\cS$ a subset of its objects, then $\sfC_{\cC}(\cS)$ is a (finite) tensor category. \qed
\end{lemma}



The result below is the main method of this paper used to construct  modular tensor categories.

\begin{theorem}\label{thm:centmodular}
 Let $\cD$ be a modular category in the sense of Definition \ref{def:modular}, let $\cE$ be a topologizing braided tensor subcategory of $\cD$, and consider the M\"{u}ger centralizer $\sfC_{\cD}(\cE)$. Then,  $$\sfC_{\cD}(\cE)' \simeq \cE'.$$ As a consequence, $\sfC_{\cD}(\cE)$ is modular if and only if $\cE$ is modular.
\end{theorem}

\begin{proof}

First, let $V$ be an object in $\cE'$. Then, for any object $W$ in $\cE$, the equation
$
    \ide_{V\otimes W}
    =c_{W,V} \;c_{V,W}
$
holds in $\cD$ using that $\cE$ is a full braided subcategory of $\cD$.
By definition, this shows that $V$ is an object in $\sfC_{\cD}(\cE)$. Now let $X$ be an object in $\sfC_{\cD}(\cE)$. Again, by definition and since $V\in \cE$, we get that $V$ centralizes $X$. Hence, $V$ is contained in the M\"uger center $\sfC_{\cD}(\cE)'$.

Conversely, let $X$ be an object in $\sfC_{\cD}(\cE)'$. Then, using $\sfC_{\cD}(\cE)\subseteq \cD$, we have that $X$ is in the M\"uger centralizer $\sfC_{\cD}(\sfC_{\cD}(\cE)).$ Using the double centralizer theorem \cite{Shi1}*{Theorem 4.9}, it follows that 
$\sfC_{\cD}(\sfC_{\cD}(\cE))$  equals $\cE$.
Note that this result uses that $\cE$ is a topologizing subcategory of $\cD$.
Hence, $X$ is an object of $\cE$. But as $X$ was assumed to be an object of $\sfC_{\cD}(\cE)$ it centralizes all objects of $\cE$ and is thus isomorphic to an object in $\cE'$.
Therefore, $\sfC_{\cD}(\cE)'\simeq  \cE'$, as desired.

Hence, $\sfC_{\cD}(\cE)$ is non-degenerate as in Definition~\ref{def:nondeg} if and only $\cE$ is. Since both $\sfC_{\cD}(\cE)$ and $\cE$ are ribbon subcategories of $\cD$ [Lemma \ref{lem:ribbon-subcat}], the consequence holds.
\end{proof}

If $\cE$ is not a topologizing subcategory that we can replace $\cE$ by its subquotient completion, which is also a braided tensor subcategory of $\cD$, in the statement of the theorem above.
Note that if $\cD$ is semisimple, then $\cE$ is a topologizing subcategory provided that it is a full subcategory closed under direct summands. (This follows as the simple objects of $\cE$ are also simple in $\cD$ and cannot have any non-trivial subquotients.)


\subsection{\texorpdfstring{$\cB$}{B}-central monoidal categories and relative monoidal centers}
\label{sec:Bcenter}
Let $\cB := (\cB, \otimes_{\cB}, \one_{\cB}, \psi)$ be an abelian braided monoidal category   throughout this section. Also, recall the braided monoidal category $\overline{\cB} := (\cB,\; \psi^{-1}_{Y,X}\colon X \otimes Y \overset{\sim}{\to} Y \otimes X)$.

\begin{definition} \label{def:Bcentral}
A monoidal category $\cC$ is {\it $\cB$-central} if there exists a faithful braided monoidal functor $G\colon \overline{\cB} \to \cZ(\cC)$.  
In this case, we refer to the functor $G$ as {\it $\cB$-central} as well.

Likewise, if $\cB$ is a braided (finite) tensor category, then we say that a (finite) tensor category $\cC$ is {\it $\cB$-central} if there exists a faithful braided tensor functor $G \colon\overline{\cB} \to \cZ(\cC)$. 
\end{definition}

\begin{remark}  \label{rem:Bcentral}
We compare our notion of a $\cB$-central functor with similar notions in the literature.
\begin{enumerate}
\item
Denote by $F\colon \cZ(\cC)\to \cC$ the forgetful functor.
We have that for a $\cB$-central functor $G$, the functor $T:=F\circ G \colon  \overline{\cB} \to \cC$ is central in the sense of \cite{DNO}*{Definition~2.3} and, in addition, faithful.
Later in \cite{DNO} only central functors such that $T$ is fully faithful are considered. While faithfulness of $G$ is equivalent to faithfulness of $T$ we do not require that $T$ is full.

\smallskip

\item Recall from \cite{L18}*{Section~3.3} that a monoidal category $\cC$ is {\it $\cB$-augmented} if it comes equipped with monoidal functors  
$F'\colon \cC \to \cB$ and $T'\colon \cB \to \cC$ 
and natural isomorphisms
$\tau\colon F'T' \isomorph \ide_\cB$ and $\sigma\colon \otimes_\cC (\ide_\cC \boxtimes T') \isomorph \otimes_\cC^{\oop} (\ide_\cC \boxtimes T')$
such that $\sigma$ descends to $\psi$ under $F'$, $\tau$ and $\sigma$ are coherent with the structure of $\cC$ and $\cB$. So a $\cB$-augmented monoidal category $\cC$ is $\cB$-central. 
In fact, we may define a functor of braided monoidal categories $G\colon \overline{\cB}\to \cZ(\cC)$, by  $B\mapsto (T'(B),\sigma^{-1})$ and by $T'$ on morphism spaces. Since $T'$ has a right inverse, it is faithful, and thus $G$ is faithful.

\smallskip

\item Note that, in the semisimple case, if $\cB$ is non-degenerate, then $G$ is fully faithful; see \cite{DMNO}*{Corollary~3.26}.
\end{enumerate}
\end{remark}

\begin{example}\label{exmpl:br-Hopf}
Given $H$ a Hopf algebra in a braided monoidal category (respectively, braided (finite) tensor category) $\cB$, we have that $\cC=\lmod{H}(\cB)$ is a $\cB$-central monoidal category (respectively, $\cB$-central (finite) tensor category) \cite{L18}*{Example 3.17}. Define a braided monoidal functor $G\colon \ov{\cB}\to \cZ(\cC)$ by sending $V\in \ov{\cB}$ to $((V,a_V^{\triv}),\psi^{-1}_{-,V})$,  where 
$a_V^\triv:=\varepsilon\otimes \ide_V\colon H \otimes V \to V$
is the trivial $H$-action on $V$ and $\psi$ is the braiding of $\cB$.
The two conditions of the action being trivial, and the half-braiding equaling $\psi^{-1}$, are stable under taking subquotients in $\cZ(\cC)$. Hence, the image of $\ov{\cB}$ in $\cZ(\cC)$ is a topologizing subcategory. 
\end{example}

\begin{definition} \label{def:relcenter}
Given a $\cB$-central monoidal category $\cC$, we define the \emph{relative monoidal center} $\cZ_\cB(\cC)$ to be the braided monoidal full subcategory consisting of objects $(V,c)$ of $\cZ(\cC)$, where $V$ is an object of $\cC$, and the half-braiding $c:=c_{V,-}\colon  V \otimes \ide_{\cC} \isomorph \ide_{\cC} \otimes V$ is a natural isomorphism satisfying the two conditions below:
\begin{enumerate}
\item[(i)] [tensor product compatibility] \; $c_{V,X \otimes Y} = (\ide_X \otimes c_{V,Y})(c_{V,X} \otimes \ide_Y)$, for $X,Y \in \cC$.
\smallskip
\item[(ii)] [compatibility with $\cB$-central structure] \; $c_{G(B),V} \circ c_{V,G(B)} = \ide_{V \otimes G(B)}$, for any $B \in \overline{\cB}$.
\end{enumerate}
That is, $\cZ_\cB(\cC)$ is the full subcategory of $\cZ(\cC)$ of all objects that centralize $G(B)$ for any object $B$ of $\ov{\cB}$.
\end{definition}

\begin{remark} \label{rem:ZBC}
 It is clear, by definition, that $\cZ_\cB(\cC)$ is equal to the M\"{u}ger centralizer ${\sf C}_{\cZ(\cC)}(G(\overline{\cB}))$. Hence, $\cZ_\cB(\cC)$ is a topologizing braided monoidal subcategory of $\cZ(\cC)$. 
\end{remark}

\begin{proposition} \label{prop:ZBCfinite-tens}
Take $\cB$ a braided (finite) tensor category, and $\cC$ a $\cB$-central (finite) tensor category (with $\cB$-central functor $G\colon  \overline{\cB} \to \cZ(\cC)$). Then, $\cZ_\cB(\cC)$ is a braided (finite) tensor category. 
\end{proposition}

\begin{proof}
First, recall that $\cZ(\cC)$ is a braided finite tensor category by Proposition~\ref{prop:ZC-finitetens}. Next, Remark~\ref{rem:ZBC} implies that the braiding in $\cZ_\cB(\cC)$ is the restriction of the braiding in $\cZ(\cC)$. Finally, with Lemma~\ref{lem:cent-finite}, the full braided subcategory $\cZ_\cB(\cC)=\mathsf{C}_{\cZ(\cC)}(G(\ov{\cB}))$ of $\cZ(\cC)$ is a tensor category, and it is finite provided that $\cC$ is finite.
\end{proof}

\begin{remark}
An important invariant of a finite tensor category $\cC$ is the \emph{Frobenius--Perron dimension} $\FPdim(\cC)$ defined in \cite{EGNO}*{Definition 6.1.7}. It follows directly from  \cite{Shi1}*{Theorem~4.9} and \cite{EGNO}*{Theorem 7.16.6}
for $\cC$ a $\cB$-central finite tensor category, the Frobenius--Perron dimension of the relative center is given by
$$\FPdim \big(\cZ_\cB(\cC)\big)= \frac{\FPdim (\cC)^2}{\FPdim(\cB)}.$$
\end{remark}

\begin{example}\label{expl:trivial}
For any braided monoidal category $\cB$ consider the functor $\ov{\cB}\to \cZ(\cB)$ of braided monoidal categories given by sending 
$X$ to $(X,\psi^{-1}_{-,X})$. This functor makes $\cB$ a $\cB$-central monoidal category such that $\cZ_\cB(\cB)\simeq \cB$ as braided monoidal categories.
\end{example}

\begin{example}\label{expl:br-YD}
In the setup of Example \ref{exmpl:br-Hopf}, where $\cB$ is rigid and $\cC = \lmod{H}(\cB)$ for $H \in \HopfAlg(\cB)$, we have an equivalence of braided monoidal categories 
$$\cZ_\cB(\cC)\simeq \lYD{H}(\cB),$$
see \cite{L18}*{Proposition 3.36}.  If $\cB$ is a (finite) tensor category here, then so is $\cZ_\cB(\cC)$.
\end{example}

\begin{example}\label{expl:br-Drin}
Continuing the example above, let $\cB=\lmod{K}$ for a finite-dimensional quasi-triangular Hopf algebra $K$, and $H$ a finite-dimensional Hopf algebra in $\cB$ with dual $H^*$ (as in dually paired Hopf algebras \cite{L17}*{Definition 3.1}), and $$\cC=\lmod{H}(\lmod{K})\simeq \lmod{H\rtimes K}.$$ Then there is an equivalence of tensor categories 
$$\cZ_\cB(\cC)\simeq \lYD{H}(\cB)\simeq \lmod{\Drin_K(H,H^*)}.$$
Here, $\Drin_K(H,H^*)$ is a quasi-triangular Hopf algebra called the \emph{braided Drinfeld double} of $H$. It is due to \cite{Maj99} where it is referred to as the \emph{double bosonization}. For details, including a presentation of $\Drin_K(H,H^*)$, see \cite{L17}*{Section 3.2}. See Lemma \ref{NicholsBD-pres} for a presentation in the case that $H$ is a Nichols algebra of diagonal type. The case when $K = \Bbbk$ is discussed at the end of Section~\ref{sec:Hmod}.
\end{example}

\subsection{Main application: Modularity of \texorpdfstring{$\cZ_\cB(\cC)$}{ZB(C)}}
\label{sec:Bcentralmodular}
In this part, we provide the main application of Theorem \ref{thm:centmodular} to establish when a  relative monoidal center is modular. This result provides sufficient conditions that $\cZ_\cB(\cC)$ is a modular category.

\begin{theorem}  \label{thm:ZBCmodular}
 Let $\cB$ be a non-degenerate braided finite tensor category, and $\cC$ a $\cB$-central  
 finite tensor category so that the set   $\mathsf{Sqrt}_\cC(D,\xi_{D})$ from Definition \ref{def:sqrtset} is non-empty. Assume that the full image $G(\ov{\cB})$ in $\cZ(\cC)$ is a topologizing subcategory. Then the relative monoidal center $\cZ_\cB(\cC)$ is a modular tensor category.
 \end{theorem}

 Here, the braiding of $\cZ_\cB(\cC)$ is restricted from that of $\cZ(\cC)$, see  Proposition~\ref{prop:ZBCfinite-tens}, and the ribbon structure is restricted to $\cZ_\cB(\cC)$ from that of $\cZ(\cC)$ by Lemma~\ref{lem:ribbon-subcat}.

\begin{proof}[Proof of Theorem \ref{thm:ZBCmodular}]
We have that  $\cZ(\cC)$ is modular by Theorem~\ref{thm:Shimizucenter}. By assumption, the full image $G(\ov{\cB})$ is a topologizing subcategory of $\cZ(\cC)$, and since  $G\colon \ov{\cB} \to \cZ(\cC)$ is faithful, $G(\ov{\cB})$ is non-degenerate. Now, apply Theorem~\ref{thm:centmodular} with $\cD = \cZ(\cC)$ and  $\cE = G(\ov{\cB})$, together with Remark~\ref{rem:ZBC}, to conclude that $\cZ_\cB(\cC)$ is modular, as desired.
\end{proof}

\begin{remark}
The statement of Theorem \ref{thm:ZBCmodular} can be varied to requiring that the subquotient completion of the full image of $G$ (with braiding obtained from being a tensor subcategory of $\cZ(\cC)$) is non-degenerate instead of requiring that $\cB$ is non-degenerate  and $G(\ov{\cB})$ topologizing. 
\end{remark}

In the setting of Example~\ref{expl:br-YD}, Shimizu achieved Theorem~\ref{thm:ZBCmodular} for $\cZ_{\cB}(\cC)\simeq \lYD{H}(\cB)$ in
\cite{Shi3}*{Theorem~4.2} using \cite{Shi1}*{Theorem~6.2}. 
Next, by Theorem~\ref{thm:Shimizucenter} we obtain the following result as a special case.

\begin{corollary}  \label{cor:ZBCmodular}
 Let $\cB$ be a non-degenerate braided finite tensor category, and $\cC$ a $\cB$-central 
 finite tensor category which is spherical in the sense of Definition \ref{def:spherical} and such that the full image $G(\ov{\cB})$ in $\cZ(\cC)$ is a topologizing subcategory. Then the relative monoidal center $\cZ_\cB(\cC)$ is modular.
 \end{corollary}

\begin{proof} 
This follows directly from Theorem \ref{thm:ZBCmodular} using Remark~\ref{rem:sph-rib}. 
\end{proof}

Observe that this result is a relative generalization of Corollary~\ref{cor:Shimizucenter} above due to Shimizu.

\subsection{M\"{u}ger's decomposition theorem in the non-semisimple setting} \label{sec:decomposition}

In this section, we establish a generalization of M\"{u}ger's decomposition theorem of modular fusion categories \cite{Mue}*{Theorem~4.5} to the non-semisimple setting [Corollary~\ref{cor:primedecomp}].
We begin with a non-semisimple generalization of \cite{Mue}*{Theorem~4.2}.

\begin{theorem}\label{thm:decomp}
Let $\cD$ be a modular tensor category, with a topologizing non-degenerate braided tensor subcategory $\cE$. Then there is an equivalence of ribbon categories:
$$\cD \simeq \cE \boxtimes {\mathsf C}_\cD(\cE).$$
\end{theorem}

Here, $\cE$ is modular [Remark~\ref{rem:modular-subcat}]; the results in \cite{Mue} have $\cE$ being modular as a hypothesis. 

\begin{proof}[Proof of Theorem~\ref{thm:decomp}]
As in the discussion before \cite{Shi1}*{Lemma~4.8},  let $\cD_1$ and $\cD_2$ be topologizing subcategories of $\cD$, and let $T\colon  \cD_1 \boxtimes \cD_2 \to \cD$ be the functor induced by $\otimes^{\cD}$. Set $\cD_1 \vee \cD_2$ to be the closure under subquotients of the image of $T$. We then get that $\cD_1 \vee \cD_2$ and $\cD_1 \cap \cD_2$ are topologizing full subcategories of $\cD$. Applying this to  $\cD_1 = \cE$ and $\cD_2 = {\mathsf C}_\cD(\cE)$ we see that
$$\cE \cap {\mathsf C}_\cD(\cE) = \{V \in \cE ~|~ c_{W,V} \; c_{V,W} = \ide_{V \otimes W} \; \forall W \in \cE\} = \cE' \simeq \Vect.$$ 

Next, $$\cE \boxtimes {\mathsf C}_\cD(\cE) \simeq \cE \vee {\mathsf C}_\cD(\cE)$$ as ribbon tensor categories via the functor $T$ above. Indeed we see  by construction that  $T$ is essentially surjective. Applying \cite{Shi1}*{Lemma~4.8} with $\cE \cap {\mathsf C}_\cD(\cE) \simeq \Vect$ we get that \begin{center}
    $\FPdim(\cE \vee {\mathsf C}_\cD(\cE))$ = $\FPdim(\cE)$\;$\FPdim({\mathsf C}_\cD(\cE))$ = 
    $\FPdim(\cE \boxtimes {\mathsf C}_\cD(\cE))$.
\end{center} 
So by \cite{EGNO}*{Proposition~6.3.4}, $T$ is an equivalence  of categories. Moreover, $\cE$ and ${\mathsf C}_\cD(\cE)$ centralize each other by definition. So, as in the proof of \cite{Mue2}*{Proposition~7.7}, $T$ is a functor of braided tensor categories such that $T(\theta_V\boxtimes \theta_W)=\theta_{V\otimes W}$, for all $V\in \cE$, $W\in {\mathsf C}_\cD(\cE)$. Thus, $T$ is an equivalence of ribbon tensor categories as the braiding and twist are preserved.

Finally, we get the desired result by the  computation below, which follows from the double centralizer theorem \cite{Shi1}*{Theorem~4.9}: 
\begin{align*}
\cE \vee {\mathsf C}_\cD(\cE) 
= {\mathsf C}_\cD[{\mathsf C}_\cD(\cE \vee {\mathsf C}_\cD(\cE))] 
&= {\mathsf C}_\cD[{\mathsf C}_\cD(\cE) \cap{\mathsf C}_\cD({\mathsf C}_\cD(\cE))]\\
&= {\mathsf C}_\cD[{\mathsf C}_\cD(\cE) \cap \cE]
\simeq {\mathsf C}_\cD(\Vect) = \cD.
\qedhere
\end{align*}
\end{proof}

\begin{example} \label{ex:decomp1}
Let $\cB$ be a non-degenerate braided finite tensor category, and let $\cC$ be a $\cB$-central finite tensor category, with $\cB$-central functor $G\colon  \ov{\cB} \to \cZ(\cC)$. Assume that $G(\ov{\cB})$ is a topologizing subcategory of $\cZ(\cC)$ and that $\mathsf{Sqrt}_\cC(D,\xi_{D}) \neq \varnothing$.
\begin{enumerate}
    \item Then, by the Theorems~\ref{thm:ZBCmodular} and~\ref{thm:decomp}, we have a decomposition of modular tensor categories:  $$\cZ(\cC)\; \simeq \; \ov{\cB} \boxtimes \cZ_\cB(\cC).$$
    \item If, further, $\cB = \lmod{K}$ and $\cC = \lmod{H}(\lmod{K})$, for a quasi-triangular Hopf algebra $K$, and for a finite-dimensional Hopf algebra $H$ in $\cB$, then by Example~\ref{expl:br-Drin}, we have a decomposition of modular tensor categories:  $$\lmod{\Drin(H \rtimes K)}\; \simeq \; \lmod{K} \;\boxtimes \;\lmod{\Drin_K(H,H^*)}.$$
\end{enumerate}
\end{example}

In comparison with \cite{Mue}*{Definition~4.4}, consider the following terminology.

\begin{definition} \label{def:prime} A modular tensor category $\cC$ is {\it prime} if every topologizing non-degenerate braided tensor subcategory is equivalent to either $\cC$ or $\Vect$.
\end{definition}

As a consequence of Theorem~\ref{thm:decomp}, we immediately obtain the result below; cf., \cite{Mue}*{Theorem~4.5}.

\begin{corollary} \label{cor:primedecomp}
Every modular tensor category is equivalent to a finite Deligne tensor product of prime modular categories. \qed
\end{corollary}

\begin{question} \label{ques:DKH}
Continuing Example~\ref{ex:decomp1}(2), when is the (not necessarily semisimple) modular category $\lmod{\Drin_K(H,H^*)}$ prime?
\end{question}

In Example~\ref{ex:uqsl2} below, we recall that the (non-semisimple, factorizable, ribbon) small quantum group $u_q(\mathfrak{sl}_2)$ arises as a braided Drinfeld double $\Drin_K(H,H^*)$. It is an interesting question to determine the primality of its module category, and as well as of the  module categories of other examples of non-semisimple braided Drinfeld doubles in the  next section. In the semisimple case, M\"uger offers examples of prime and non-prime module categories over Drinfeld doubles of groups \cite{Mue}*{Theorem 4.7, Table 1}.


\section{Examples of modular categories}\label{sec:examples}

In this part, we provide several examples of modular tensor categories using the relative center construction  [Theorem \ref{thm:ZBCmodular}]. In some cases, we also illustrate our decomposition result Theorem~\ref{thm:decomp} above. We start in Section~\ref{sec:firstex} by discussing relative centers over $\Vect$, monoidal centers of modules over Taft algebras, and their relation to $\lmod{u_q(\fr{sl}_2)}$. Next, we provide preliminary information about braided doubles of Nichols algebras  of diagonal type in Section~\ref{sec:Nichols}. We study the modularity of braided doubles of such Nichols algebras  in Section~\ref{sec:braided-mod}; the main result is Proposition~\ref{cor:UBqmod} there. Finally, in Section~\ref{ex:smallqg} we apply this result to module categories of small quantum groups (of Cartan type) and also to module categories of a braided Drinfeld double of Nichols algebras not of Cartan type. Throughout this section, we additionally assume that $\Bbbk$ has characteristic zero.

\subsection{First examples} \label{sec:firstex}

Here, we include some first examples of non-semisimple modular categories obtained from the general result of Theorem \ref{thm:ZBCmodular}.

\begin{example} \label{ex:Vec}
Take $\cB=\Vect$, the category of finite-dimensional $\Bbbk$-vector spaces  with its usual symmetric structure. Let $\cC$ be a finite tensor category over $\Bbbk$ (i.e., $\cC$ is $\Vect$-central) so that the set $\mathsf{Sqrt}_\cC(D,\xi_{D})$ is non-empty. For instance, if $\cC$ is unimodular finite ribbon category, then $\cC$ is spherical [Proposition~\ref{prop:ribbonspherical}] and hence $\mathsf{Sqrt}_\cC(D,\xi_{D}) \neq \varnothing$ [Remark \ref{rem:sph-rib}]. Then Theorem~\ref{thm:ZBCmodular} specializes to the result that $\cZ_{\Vect}(\cC)=\cZ(\cC)$ is  modular, recovering Theorem~\ref{thm:Shimizucenter}.

\end{example}

\begin{example}[(Drinfeld double of) the Taft algebra $T_{n}(q^{-2})$]\label{expl:taft} 
Let $n\geq 3$ be an integer, and let $q$ be a primitive root of unity so that $q^2$ has order $n$. Take $K=\Bbbk \mathbb{Z}_n$, for $\mathbb{Z}_n = \langle g ~|~ g^n =1 \rangle$, and set $\cB_q=\lmod{K}$. Here, $\cB_q$ is braided using the $R$-matrix
$R = \frac{1}{n} \sum_{i,j = 0}^{n-1} q^{-2ij}g^i \otimes g^j.$
Next, take the monoidal category
$$\cC:= \lmod{H}{(\cB_q)},  \; \;  \text{for } H:=\Bbbk[x]/(x^n) \in \HopfAlg(\cB_q),$$ 
with $\Delta(x) = x\otimes 1+1\otimes x$ and $\varepsilon(x) = 0$. As in  \cite{LW}*{Section~6}, consider the  \emph{Taft algebra} $T_n(q^{-2})$, which is the $\Bbbk$-Hopf algebra: 
$$T:=T_n(q^{-2}) =  \Bbbk \langle g,x \rangle / (g^n-1, ~x^n, ~gx - q^{-2} xg),$$
with $\Delta(g) = g \otimes g, ~\Delta(x) = g^{-1} \otimes x + x \otimes 1, ~\varepsilon(g) = 1, ~\varepsilon(x) =0, ~S(g) = g^{-1}, ~S(x) = -gx$. Then, we get an  equivalence of monoidal categories:
$$\cC \simeq \lmod{T_n(q^{-2})}.$$

Computations as in \cite{KR93}*{Proposition~7} show that if $n$ is even, then $\mathsf{Sqrt}_\cC(D,\xi_{D}) =\varnothing$; see Proposition~\ref{prop:KRvsShi}. If $n=2m-1$ is odd, then  the distinguished grouplike elements (see \eqref{eq:alpha_H}, \eqref{eq:g_H}) are given by 
$g_T=g$ and $\alpha_T\colon T\to \Bbbk$ with $\alpha_T(g)=q^{-2}$, $\alpha_T(x)=0.$
Thus, $a=g_T^{m}$ and $\zeta=\alpha_T^{m}$ are the unique elements satisfying the equations of Theorem \ref{thm:KR}. Hence, using Proposition
\ref{prop:KRvsShi}, the unique element $(V,\sigma_{V})$ of $\mathsf{Sqrt}_\cC(D,\xi_{D})$ is given by the $1$-dimensional $T$-module $V=\Bbbk v$, where $g\cdot v=q^{2m}$, $x\cdot v=0$, and
$\sigma_{V}(W)\colon V\otimes W\isomorph W^{**}\otimes V$, with $\sigma_{V}(W)(v\otimes w)= (g^m \cdot w)\otimes v$, for $W\in \Ob \cC$, $w\in W$.
Thus, $\mathsf{Sqrt}_\cC(D,\xi_{D})\neq \varnothing$ if and only if $n$ is odd.
In this case,  by Theorem~\ref{thm:Shimizucenter}, we get that $\cZ(\cC)$ is modular with a unique ribbon structure. We further have an equivalence of modular categories
$$\cZ(\cC) \simeq \lmod{\text{Drin}(T_n(q^{-2}))},$$
where modularity of the right-hand side is inherited from that on the left-hand side. We note that $\cC$ is \emph{not} a spherical category in the sense of Definition \ref{def:spherical} since $D\not\cong \one$, cf. \cite{DSS}*{Section 3.5.2} for the case $n=3$. In particular, $\alpha_T$ fails to equal $\varepsilon$ as in Proposition~\ref{prop:Hspherical}.
\end{example}

\begin{example}[The small quantum group $u_q(\mathfrak{sl}_2)$] \label{ex:uqsl2}
For $q$ as in Example \ref{expl:taft}, with $n$ odd, consider the small quantum group $u_q(\mathfrak{sl}_2)$ which is generated by  $k$, $e$, $f$, subject to the relations
$$k^n =1, \quad e^n = f^n = 0, \quad ke = q^2 ek, \quad kf = q^{-2}fk, \quad ef-fe = \frac{k - k^{-1}}{q-q^{-1}},$$
with coproduct and counit determined on generators by  
\begin{gather*}
\Delta(k) = k \otimes k, \qquad \Delta(e) = 1 \otimes e + e \otimes k,  \qquad \Delta(f) = k^{-1} \otimes f + f \otimes 1, \\
\varepsilon(k) = 1,  \qquad \varepsilon(e) = \varepsilon(f) =0.
\end{gather*}
For the categories $\cB_q$ and $\cC$ as in Example~\ref{expl:taft}, the relative center $\cZ_{\cB_q}(\cC)$ is equivalent as a monoidal category to $\lmod{u_q(\fr{sl}_2)}$, see \cite{LW}*{Proposition 6.7(i)}. Thus, $\lmod{u_q(\fr{sl}_2)}$ is a braided category.
We have seen in Example \ref{expl:taft} that since $n$ is odd, $\mathsf{Sqrt}_\cC(D,\xi_{D})$ contains a unique element. Moreover, in this case one checks that the underlying braided category $\cB_q$ is non-degenerate using \cite{EGNO}*{Exercise~8.6.4}. Thus, by Theorem \ref{thm:ZBCmodular}, $\cZ_{\cB_q}(\cC)\simeq \lmod{u_q(\fr{sl}_2)}$ is modular. 

Continuing Example~\ref{ex:decomp1}, we see that there is a decomposition of modular categories
$$\lmod{\Drin(T_n(q^{-2}))}\simeq \cB_q\boxtimes \lmod{u_q(\fr{sl}_2)}.$$
\end{example}

A vast generation of this example, using braided Drinfeld doubles of Nichols algebras of diagonal type, will be given in the following sections.

\subsection{Braided Drinfeld doubles of Nichols algebras of diagonal type}\label{sec:Nichols}
 In this part, we discuss Nichols algebras of diagonal type, a large class of Hopf algebras in braided categories $\lcomod{\Bbbk G}$, for $G$ a finite abelian group. To start, consider the following notation that will be used throughout the rest of the section.
 
\begin{notation}[$G$, $g_i$, $m_i$, $\Lambda$, $e_i$,  ${\bf i}$, $g_{\bf i}$,  $\delta_{\bf i}$, $K$, $r$, $b$, $\cB_{\qs}$, $\gamma_{\bf i}$, $\gamma_i$, $\ov{\gamma}_{\bf i}$, $\ov{\gamma}_i$] 
\label{not:Cq-etc} Fix the  notation below.
\begin{itemize}
    \item Let $G=\langle g_1,\ldots, g_n \rangle$ be a finite abelian group, where $g_i$ has order $m_i$.
    
     \smallskip
     
    \item Take $\Lambda$ to denote the lattice $\mZ_{m_1}\times \cdots \times \mZ_{m_n}$, and let $e_i$ be the $i$-th elementary vector.
    
    \smallskip
    
    \item For ${\bf i}=(i_1,\ldots, i_n)\in \Lambda$, write $g_{\bf i}:=g_1^{i_1}\ldots g_n^{i_n}$ and use additive notation on indices ${\bf i}$, so e.g. $g_{-e_i}=g_{i}^{-1}$.
    
    \smallskip

    \item Take $K$ to be the group algebra $\Bbbk G$.
    
    \smallskip 
    
    \item Denote the basis of $K^*$  dual to $\{g_{\bf i}\}$ by $\{\delta_{\bf i}\}$, so that the pairing of $K^*$ and $K$ is given by $\langle \delta_{\bf i},g_{\bf j}\rangle=\delta_{{\bf i},{\bf j}}$. 
    
    \smallskip
    
    \item Let $\qs= (q_{ij}) \in \Mat_n(\Bbbk)$ with $q_{ij} \neq 0$, and let 
    $\cB_{\qs}$ be the braided category $\lcomod{K}_r$ with dual R-matrix $r$ of $K$ given by $r(g_i \otimes g_j) = q_{ji}$.
    
    \smallskip 
    
    \item Consider the symmetric bilinear form $b$ on $K$ given by $b(g_{\bf i},g_{\bf j}):=r(g_{\bf i},g_{\bf j})r(g_{\bf j},g_{\bf i})$. In particular, $b$ is determined by $b(g_i,g_j)=q_{ij}q_{ji}.$
    
     \smallskip
    
    \item For ${\bf i}=(i_1,\ldots, i_n) \in \Lambda$, take the grouplike elements of $K^*$: 
$\gamma_{\bf i}:=\sum_{\bf j}r(g_{\bf j}\otimes g_{\bf i})\delta_{\bf j}$  and  $\ov{\gamma}_{\bf i}:=\sum_{\bf j}r(g_{\bf i}\otimes g_{\bf j})\delta_{\bf j}$. We write $\gamma_i := \gamma_{e_i}$ and $\ov{\gamma}_i := \ov{\gamma}_{e_i}$. Then $\gamma_{\bf i}=\gamma_i^{i_1}\ldots \gamma_i^{i_n}$ and $\ov{\gamma}_{\bf i}=\ov{\gamma}_i^{i_1}\ldots \ov{\gamma}_i^{i_n}$.
\end{itemize}
\end{notation}

We record a fact about non-degeneracy that will be used several times later.

\begin{lemma}[\cite{EGNO}*{Example 8.13.5}, \cite{DGNO}*{Section 2.11}] \label{lem:Bq-nondeg}
The braided category $\cB_{\qs}$ is non-degenerate if and only if the symmetric pairing $b$ is non-degenerate. \qed
\end{lemma}

Next, we recall the definition of a Nichols algebra of diagonal type from work of Andruskiewitsch--Schneider \cite{AS}.

\begin{definition}[$\BB(V)$, $\II(V)$]\label{def:Nichols}
Retain the notation above.
Let $V$  be an object in $\lYD{K}$ (with braiding $c$). Recall that the tensor algebra $\TT(V)=\bigoplus_{n\geq 0} V^{\otimes n}$ is a natural Hopf algebra object in $\lYD{K}$, such that all elements $v\in V$ are primitive, i.e., $\Delta(v)=v\otimes 1+ 1\otimes v$. 
\begin{enumerate}
   \item The \emph{Nichols algebra} $\BB(V)$ is the quotient of $\TT(V)$ be the unique largest homogeneously generated Hopf ideal $\II(V)\subseteq \bigoplus_{n>1}V^{\otimes n}$. 
    
    \smallskip
    
   \item We say that $\BB(V)$ is of \emph{diagonal type} if there exists a basis $x_1, \dots, x_n$ of $V$ so that $$c(x_i\otimes x_j)=q_{ij}x_j\otimes x_i.$$
 \end{enumerate}
\end{definition}

There exists a complete classification of finite-dimensional Nichols algebras of diagonal type over a field $\Bbbk$ or characteristic zero \cite{Hec2}. The Nichols algebras $\BB(V)$ have a PBW basis \cite{Kha} and generalized root systems \cite{Hec1}. For finite-dimensional Nichols algebras of diagonal type, relations for the ideal $\II(V)$ were found in \cites{Ang1,Ang2} and are detailed in many examples in \cite{AA}.

\begin{lemma}[$\BB_\qs$, $\BB_\qs^*$] \label{lem:Bq}
Consider the Yetter--Drinfeld module $V$ over $K$ with action $g_i\cdot x_j=q_{ij}x_j$ and coaction $\delta(x_i)= g_i \otimes x_i$. Then:
\begin{enumerate}
    \item[\textnormal{(1)}] The Nichols algebra, $\BB(V) \in \HopfAlg(\lYD{K})$ is of diagonal type, which we denote by $\BB_{\qs}$.
    \smallskip
    
    \item[\textnormal{(2)}] If $\BB_\qs$ is finite-dimensional, then $\BB_\qs$ and $\BB_\qs^*$ are dually paired Hopf algebras in $\cB_{\qs}$. In particular, the pairing $\ev\colon \BB_\qs^*\otimes \BB_\qs\to \Bbbk$ is uniquely induced from the pairing of $V^*$ and $V$.
\end{enumerate}
\end{lemma}

\begin{proof}
Part (1) follows from \cite{AS}*{Section~2}. 

For part (2), consider the braided monoidal functor  $\Phi_r\colon \cB_{\qs}\to \lYD{K}$ sending $(V,\delta_V)$ to $(V,a_V, \delta_V)$, where
$a_V(k\otimes v)=r(v^{(-1)}\otimes k)\otimes v^{(0)}$, for all $k \in K$ and $v \in V$. Then the Yetter--Drinfeld module $V$ is the image of the $K$-comodule $(V,\delta)$ under $\Phi_r$ and $\BB_\qs$ is a Hopf algebra in the image of $\Phi_r$. Thus, $\BB_\qs$, forgetting the $K$-action, is a Hopf algebra in $\cB_{\qs}$.

It is well-known that the duality pairing $\ev\colon V^*\otimes V\to \Bbbk$  extends to a unique non-degenerate pairing $\ev\colon \BB(V^*)\otimes \BB(V)\to \Bbbk$ of Hopf algebras in $\lYD{K}$, see e.g. \cite{HS}*{Corollary~7.2.8}. By the above observations, this is a non-degenerate pairing of Hopf algebras in $\cB_{\qs}$.
\end{proof}

\smallskip

Having viewed $\BB_\qs$ and $\BB_\qs^*$ as braided Hopf algebras in $\cB_{\qs}=\lcomod{K}_r$, we are able to compute their braided Drinfeld double over the Hopf algebra $K^*$, cf. Example \ref{expl:br-Drin}; see \cite{L17}*{Section 3.2} for the presentation of general braided Drinfeld doubles used here. We fix the following notation.

\begin{notation}[$x_i$, $y_i$]  For $V$ in Definition~\ref{def:Nichols} and in Lemma~\ref{lem:Bq} above, we
fix dual bases $x_1,\ldots, x_n$ of $V$, and $y_1,\ldots, y_n$ of $V^*$, and denote the resulting generators of $\BB(V)$, respectively, of the dual Nichols algebra $\BB(V^*)$ by the same symbols.
\end{notation}

\begin{proposition}[$\Drin_{K^*}(\BB_\qs^*,\BB_\qs)$]\label{NicholsBD-pres}
Retain the notation above and assume that the braided Hopf algebra $\BB_\qs$ in $\lcomod{K}_r$ from Lemma~\ref{lem:Bq} is finite-dimensional. Then the braided Drinfeld double $\Drin_{K^*}(\BB_\qs^*,\BB_\qs)$  is  a Hopf algebra generated as an algebra by elements $\{\delta_{\bf i}\}_{{\bf i} \in \Lambda}$, $\{x_i\}_{i=1,\ldots, n}$, and $\{ y_i\}_{i=1,\ldots, n}$, subject to the relations $\II(V)$ and $\II(V^*)$, along with
\begin{gather*}
\delta_{\bf i}\delta_{\bf j}=\delta_{{\bf i},{\bf j}}\delta_{\bf i},\qquad 
    \delta_{\bf i}x_j=x_j \delta_{{\bf i}-e_j}, \qquad    \delta_{\bf i}y_j=y_j \delta_{{\bf i}+e_j}, \qquad 
    y_ix_j-q_{ji}^{-1}x_jy_i=(1-\gamma_i\ov{\gamma_i})\delta_{i,j}.
\end{gather*}
Here, it is understood that $\delta_{\bf i}= \delta_{\bf j}$ if  ${\bf i}={\bf j} \in \Lambda$.
The coproduct and counit are determined by 
\begin{gather*}
    \Delta(\delta_{\bf i})=\sum_{{\bf a}+{\bf b}={\bf i}}\delta_{\bf a}\otimes \delta_{\bf b}, \qquad \Delta(x_i)=x_i\otimes 1+\gamma_i\otimes x_i, \qquad \Delta(y_i)=y_i\otimes 1+\ov{\gamma}_i\otimes y_i,\\
    \varepsilon(\delta_{\bf i})=\delta_{{\bf i},0}, \qquad \varepsilon(x_i)=\varepsilon(y_i)=0.
\end{gather*}
The antipode is determined by 
\begin{gather*}
S(\delta_{\bf i})= \delta_{\bf -i},\qquad    S(x_i)=-\gamma_i^{-1}x_i, \qquad  S(y_i)=-\ov{\gamma}_i^{-1}y_i. 
\end{gather*} 
The Hopf algebra $\Drin_{K^*}(\BB_\qs^*,\BB_\qs)$ is quasi-triangular with R-matrix given by
$$R_{\Drin}=\sum_{\alpha}\sum_{\bf i} \delta_{\bf i}y_{\alpha}\otimes x_\alpha \gamma_{\bf i},$$
where $\alpha$ indexes a basis $\{x_\alpha\}$ of $\BB_\qs$ with dual basis  $\{y_\alpha\}$ of $\BB_\qs^*$ with respect to the paring $\ev$.
\end{proposition}

\begin{proof}
First, note that for $(K,r)$ coquasi-triangular, $((K^*)^{\cop}, r^*)$ is a quasi-triangular Hopf algebra, where $(K^*)^\cop=(K,m,u,\Delta^{\oop},
\varepsilon,S^{-1})$ is the co-opposite Hopf algebra with R-matrix given by the dual of $r$, i.e.
$$R:=r^*=(r\otimes \ide_{K^*\otimes K^*})(\ide_{K}\otimes \coev_K\otimes\ide_{K^*})\coev_K.$$
The functor 
$$\Phi\colon \lcomod{K}_r\longrightarrow \lmod{(K^*)^{\cop}}_{R}, \quad (V,\delta)\longmapsto (V,a^*_V), \quad a^*_V:=(\ev_K\otimes \ide_V)(\ide_{K^*}\otimes \delta),$$
defines an equivalence of braided tensor categories. In the case $K=\Bbbk G$, $(K^*)^{\cop}=K^*$ and, thus, $\BB_\qs^*$ and $\BB_\qs$ can be regarded as dually paired Hopf algebras in $\lmod{K^*}_{R}$.

The result now follows from specifying the presentation from \cite{L17}*{Section 3.2} to the case $\Drin_{K^*}(\BB_\qs^*,\BB_\qs)$. For this, observe that $R=r^*$ for $K^*$ is given by  $R^*=\sum_{{\bf i}}\delta_{\bf i}\otimes \gamma_{\bf i}\in K^*\otimes K^*$
and that the action and coaction of $K^*$ on $V$ are given by 
\begin{align*}
\delta_{\bf i}\cdot x_j=\delta_{{\bf i},e_j}x_j,  \qquad \delta(x_j)=\gamma_{j}\otimes x_j, \qquad \gamma_{i}\cdot x_j=q_{ij}x_j, \qquad  \ov{\gamma}_{i}\cdot x_j=q_{ji}x_j.
\end{align*}

\vspace{-.25in}

\qedhere
\end{proof}

\begin{remark}[$k_{\bf i},k_i, G'$]\label{rem:ki}
Note that $\Drin_{K^*}(\BB_\qs^*,\BB_\qs)$ is a $\mZ$-graded Hopf algebra where $\deg \delta_{\bf i}=0$, for ${\bf i}\in \Lambda$, $\deg x_i=1$ and $\deg y_i=-1$, for $i=1,\ldots,n$. It has a triangular decomposition on $\BB_\qs^*\otimes K^*\otimes \BB_\qs$. Modules over this Hopf algebra can be described as a relative monoidal center, cf. Example \ref{expl:br-Drin}.
For ${\bf i}\in \Lambda$, we denote $k_{\bf i}:=\gamma_{\bf i}\ov{\gamma}_{\bf i}$ and $k_i:=\gamma_{i}\ov{\gamma}_{i}$.
When the braided category $\cB_{\qs}$ is non-degenerate [Lemma~\ref{lem:Bq-nondeg}], $K^*$ is isomorphic to the group algebra $\Bbbk G'$, where $G'=\langle k_1,\ldots,k_n \rangle$ is isomorphic to $G$. Thus, in this case, $\Drin_{K^*}(\BB_\qs^*,\BB_\qs)$ has a triangular decomposition $\BB_\qs\otimes \Bbbk G' \otimes \BB_\qs^*$. We note neither $r$ nor $r^\oop$ are necessarily non-degenerate pairings, so $\langle \gamma_1,\ldots,\gamma_n \rangle$ and $\langle \ov{\gamma}_1,\ldots,\ov{\gamma}_n \rangle$ are, in general, proper subgroups of $G'$ (see, e.g., Example \ref{expl-new} below).
\end{remark}

The next result, relating braided Drinfeld doubles and relative centers,  will be of use later.

\begin{proposition}[{\cite{L18}*{Section 4.2}}]\label{prop:Drincenter}
Retain Notation \ref{not:Cq-etc} and set $\cC=\lmod{\BB_\qs}(\cB_{\qs})$.  Then there exists an equivalence of braided monoidal categories
$$\cZ_{\cB_{\qs}}(\cC)\isomorph \lmod{\Drin_{K^*}(\BB_\qs^*,\BB_\qs)}.$$

\vspace{-.25in}

\qed
\end{proposition}

\subsection{Modularity of braided Drinfeld doubles of Nichols algebras of diagonal type} \label{sec:braided-mod}

In this section, we present an application of Theorem \ref{thm:ZBCmodular} by providing sufficient conditions for categories of finite-dimensional modules over  $\Drin_{K^*}(\BB_\qs^*,\BB_\qs)$, i.e. over braided Drinfeld doubles of finite-dimensional Nichols algebras of diagonal type, to be modular.

\begin{notation}[$\ell$, ${\bf x}_\ell$, ${\bf i}_\ell$]
\label{not:topdegree}
Continuing Notation~\ref{not:Cq-etc} assume that the Nichols algebra $\BB_\qs$ is finite-dimensional and let $\ell$ be the the top $\mZ$-degree of $\BB_\qs$. Note that $(\BB_\qs)_\ell$ is one-dimensional \cite{AS}*{Lemma~1.12}. We choose a non-zero element ${\bf x}_\ell$ in $(\BB_\qs)_\ell$ and denote its $G$-degree by ${\bf i}_\ell$. 
\end{notation}

\begin{lemma}[{\cite{AA}*{Section~2.12}}]
\label{lem:posroots}
We have that ${\bf i}_\ell=\sum_{\beta\in \Delta_{\qs}^+}(m_\beta-1)\beta \in \Lambda$, 
where $\Delta_{\qs}^+$ is the set of positive roots of the Nichols algebra $\BB_\qs$ and $m_\beta$ is the order of the root of unity  $r(g_\beta,g_\beta)$. 
\qed
\end{lemma}

\begin{lemma} \label{lem:grouplikes}
Recall \eqref{eq:alpha_H}, \eqref{eq:g_H},
Notation~\ref{not:topdegree}, and the notation of Section~\ref{sec:Nichols}. The distinguished grouplike element for $H:=\BB_\qs \rtimes K^*$ and for $H^*$, respectively, are 
\begin{align*}\label{grouplikes}
    g_H=1\otimes \gamma_{\bf i_\ell},\qquad \alpha_H(x\otimes \delta_{\bf i})=\varepsilon(x)\delta_{{\bf i},-{\bf i}_\ell}.
\end{align*}
\end{lemma}

\begin{proof}
We use techniques from \cite{Bur}*{Section 4} by first understanding integrals of $H$. These elements can be built from integrals of $\BB_\qs$ and of $K^*$ as follows. Take a left integral of $\BB_\qs$, i.e., an element $x\in \BB_\qs$ such that $hx=\varepsilon(h)x$ for any $h\in \BB_\qs$. Then for any left integral $k$ of $K^*$, we get that $\Lambda:= (1 \otimes k)(x\otimes 1)= (k_{(1)}\cdot x) \otimes k_{(2)}$ is a left integral of $H =\BB_\qs \rtimes K^*$ \cite{Bur}*{Section 4.6}.

Since $(\BB_\qs)_0=\Bbbk 1$, it follows that $1$ is the only grouplike element of $\BB_\qs$. By self-duality of $\BB_\qs$, we conclude that the distinguished grouplike elements of $\BB_\qs$ and its dual are
$g=1\in \BB_\qs$ and $\alpha=\varepsilon\in \BB_\qs^*$, respectively. Thus, $\Lambda=\delta_{0}x_{\ell}$ is a left integral for $H$ and we compute that 
$$\Lambda x_j=0=\varepsilon(x_j)=\alpha_H(x_j\otimes 1)\Lambda, \qquad \Lambda \delta_{\bf i}=\delta_0\delta_{{{\bf i}+\bf i}_\ell}x_{\ell}=\delta_{{\bf i},-{\bf i}_\ell}\Lambda=\alpha_H(1\otimes \delta_{\bf i})\Lambda,$$
verifying  the claimed formula for $\alpha_H$ on generators. 

To find $g_H$ we observe that the elements $\delta_{\bf i}^*\in H^*$ and $x_j^*\in H^*$ satisfy the relation 
\begin{align*}
    x_j^*\delta_{\bf i}^*&=r(g_{\bf i},g_j)\delta_{\bf i}^*x_j^*, &\forall j=1,\ldots, n, \;{\bf i}\in \Lambda.
\end{align*}
A right integral for $H^*$ is given by $\lambda=x_{\ell}^*\sum_{\bf i}\delta_{\bf i}^*.$ Hence, we compute on generators
\begin{gather*}
    x_j^* \lambda = 0 = \varepsilon(x_j^*)\lambda = \ev(x_j^*, 1\otimes \gamma_{{\bf i}_\ell})\lambda, \qquad \text{for all $j=1,\ldots, n$,}\\ \delta_{\bf j}^*\lambda = \delta_{\bf j}^*\sum_{\bf i}x_{\ell}^*\delta_{\bf i}^*=\sum_{\bf i}r^{-1}(g_{\bf j},g_{-{\bf i}_\ell})x_{\ell}^*\delta_{\bf j}^*\delta_{\bf i}^*=r(g_{\bf j},g_{{\bf i}_\ell})\lambda=\ev(\delta_{\bf j}^*, 1\otimes \gamma_{{\bf i}_\ell})\lambda, \qquad \text{for all ${\bf j}\in \Lambda$}.
\end{gather*}
This computation verifies the claimed formula for $g_H$.
 \end{proof}

See \cite{AA}*{Proposition 2.42} for similar computations for $\BB_\qs\rtimes \Bbbk G$.  
Next, we derive the following conditions for $\lmod{\Drin_{K^*}(\BB_\qs^*,\BB_\qs)}$ to be modular.

\begin{proposition}\label{cor:UBqmod} Recall Notation~\ref{not:Cq-etc} and \ref{not:topdegree}, and Remark \ref{rem:ki}.
 The braided tensor categories  $$\lmod{\Drin_{K^*}(\BB_\qs^*,\BB_\qs)}\simeq \cZ_{\cB_{\qs}}(\lmod{\BB_\qs}(\cB_{\qs}))$$ are modular when
 \begin{itemize}
     \item[\textnormal{(i)}] the symmetric bilinear form $b$ on $K$ is non-degenerate, and
     \item[\textnormal{(ii)}] there exist ${\bf j},{\bf a} \in \Lambda$ such that $2{\bf j}={\bf i}_\ell$, $2{\bf a}={\bf i}_\ell$, $b(g_i,g_{\bf a})^2=r(g_i,g_{{\bf i}_\ell})$, and $r(g_{\bf j},g_i) b(g_i,g_{\bf a})=q_{ii}^{-1}$ for all $i=1,\ldots, n$.
\end{itemize}
\end{proposition}

\begin{proof}
To show that $\cZ_{ \cB_{\qs}}(\lmod{\BB_\qs}(\cB_{\qs}))$ is modular,  it suffices to check that (a) the braided finite tensor category $\cB_{\qs}$ is non-degenerate and that (b) the set ${\sf Sqrt}_{\lmod{\BB_\qs}(\cB_{\qs})}(D, \xi_{D})$ is non-empty, by Theorem~\ref{thm:ZBCmodular}. Then, the equivalent category, $\lmod{\Drin_{K^*}(\BB_\qs^*,\BB_\qs)}$ [Proposition~\ref{prop:Drincenter}], is also modular. 
Now (a) follows from (i) using Lemma \ref{lem:Bq-nondeg}.

Now we show that (ii) implies (b). Take $H:=\BB_\qs\rtimes K^*$, a finite-dimensional Hopf algebra with $\lmod{H} \simeq \lmod{\BB_\qs}(\cB_{\qs})$. By~(ii), we can define
\begin{align*}
    a=1\otimes k_{\bf a},\qquad \zeta(x\otimes \delta_{\bf i})=\varepsilon(x)\delta_{{\bf i},-{\bf j}},
\end{align*} and compute using Proposition \ref{NicholsBD-pres} that
\begin{gather*}
a^2=k_{\bf a}^2=\sum_{{\bf i},{\bf l}\in \Lambda}b(g_{\bf i},g_{\bf a})b(g_{\bf l},g_{\bf a})\delta_{\bf i}\delta_{\bf l}=\sum_{{\bf i}\in \Lambda}b(g_{\bf i},g_{\bf a})^2\delta_{\bf i}=\sum_{{\bf i}\in \Lambda}r(g_{\bf i},g_{{\bf i}_\ell})\delta_{\bf i}=\gamma_{{\bf i}_\ell}=g_H, \\\zeta^2(\delta_{\bf i})=(\zeta\otimes \zeta)(\Delta(\delta_{\bf i}))=\sum_{{\bf k}+{\bf l}={\bf i}}\zeta(\delta_{\bf k})\zeta(\delta_{\bf l})=\sum_{{\bf k}+{\bf l}={\bf i}}\delta_{{\bf k},-{\bf j}}\delta_{{\bf l},-{\bf j}}=\delta_{{\bf i},-{\bf i}_\ell}=\alpha_H(\delta_{\bf i}),\\ \zeta^2(x_i)=\zeta(x_i)\zeta(1)+\zeta(\gamma_i) \zeta(x_i)=0=\varepsilon(x_i)=\alpha_H(x_i),
\end{gather*}
for all $i=1,\ldots, n$ and ${\bf i}\in \Lambda$.
Here, $\alpha_H$ and $g_H$ are the distinguished grouplike elements of $H^*$ and $H$, respectively, of  Lemma~\ref{lem:grouplikes}. For these elements, again using Proposition \ref{NicholsBD-pres}, we see that 
\begin{gather*}
S^2(x_i) = \gamma_i^{-1}x_i\gamma_i=q_{ii}^{-1}x_i,\\
\zeta^{-1}((x_i)_{(1)})a(x_i)_{(2)}a^{-1}\zeta((x_i)_{(3)})=\zeta^{-1}(\gamma_i)ax_ia^{-1}\zeta(1)=
r(g_{\bf j},g_i) b(g_i,g_{\bf a}) x_i,
\end{gather*}
for all $i$. Here, $\gamma_i$ is from Notation \ref{not:Cq-etc} and  $\zeta^{-1}(\delta_{\bf i})=\delta_{{\bf i},{\bf j}}$.  Using that $r(g_{\bf j},g_i) b(g_i,g_{\bf a})=q_{ii}^{-1}$ from~(ii), we conclude that condition~\eqref{eq:S2}
holds for $h=x_i$. This equation is evident for $h\in G$ since $G$ is an abelian group and hence holds for all $h\in H$ using that $S^2$ is an algebra morphism. 
Thus, applying Proposition~\ref{prop:KRvsShi} with the above elements $a$ and $\zeta$ yields (b). 
\end{proof}

\begin{remark}\label{rmk:UBqmod}
 \begin{enumerate}
     \item 
  As a consequence of Radford's $S^4$-formula \cite{Rad76} for the Hopf algebra $H = \BB_\qs\rtimes K^*$, we obtain that the values $\qs=(q_{ij})$ satisfy
\begin{align*}
    r(g_{{\bf i}_\ell},g_i) r(g_i,g_{{\bf i}_\ell})=q_{ii}^{-2}.
\end{align*}
 \item Similar conditions as in Proposition \ref{cor:UBqmod}(ii) were already derived in \cite{AA}*{Proposition~2.42} to determine when $\Drin(\BB_\qs\rtimes \Bbbk G)$ is a ribbon Hopf algebra.
 
 \smallskip
 
 \item Using Proposition \ref{prop:Hspherical}, we derive necessary and sufficient conditions for $\cC=\lmod{\BB_\qs}(\cB_\qs)$ to be spherical. The monoidal category
$$\lmod{\BB_\qs}(\cB_\qs)\simeq \lmod{\BB_\qs\rtimes K^*}$$
is spherical if and only if ${\bf i}_\ell=0$ and there exist ${\bf b}, {\bf c}\in \Lambda$ such that for $a:=\gamma_{\bf b}\ov{\gamma}_{\bf c}$,
$$
    a^2=1, 
    \qquad r(g_i,g_{\bf b})r(g_{\bf c},g_i)=q_{ii}^{-1} \;\; \text{for all $i=1,\ldots, n$}.
$$
In this case, $\mathsf{SPiv}(H)$ is given by these elements $a$.
See Example \ref{expl-new} for a class of examples of spherical categories obtained this way.
 \end{enumerate}
 \end{remark}

\subsection{Small quantum groups and other examples} \label{ex:smallqg}
In this subsection, we include some specific examples and demonstrate how Theorem \ref{thm:ZBCmodular} and Proposition~\ref{cor:UBqmod} lead to examples of non-semisimple modular categories. 

\begin{example}[The small quantum group $u_q(\fr{g})$] \label{ex:uqg}
 Take $q$ a root of unity of odd order $l \geq 3$  and let $\mathfrak{g}$ be the semisimple Lie algebra of rank~$t$, associated to the irreducible symmetrizable Cartan matrix $(a_{ij})_{i,j=1}^{t}$. We choose coprime integers $d_i = 1,2,3$ so that $(d_i a_{ij})$ is a symmetric matrix. 
 Associated to this data, one defines a finite-dimensional Hopf algebra $u_q(\fr{g})$, the \emph{small quantum group} (or \emph{Frobenius--Lusztig  kernel}) as, e.g., in \cite{Ro}*{Section~3.2}.\footnote{When $\fr{g}$ is of type $G_2$, assume $l$ is coprime to $3$;  $d_i=3$ only appears in this case.} The Hopf algebras $u_q(\mathfrak{g})$ generated by group-like elements $k_i^{\pm 1}$, $(k_i,1)$-skew primitive elements $e_i$, and $(1,k_i^{-1})$-skew primitive elements $f_i$ for $i = 1,\dots t$,  subject to relations:
\begin{gather*}
k_i^l = 1, \quad \quad e_i^l = 0, \quad \quad f_i^l = 0,\\
k_ie_j = q^{d_i a_{ij}} e_j k_i, \quad \quad
k_if_j = q^{-d_i a_{ij}} f_j k_i, \quad \quad
e_if_j - f_j e_i = \delta_{i,j}(k_i - k_i^{-1})(q^{d_i} - q^{-d_i})^{-1},\\
\sum_{m=0}^{1-a_{ij}}(-1)^m \binom{1-a_{ij}}{m}_{q^{d_i}}\; e_i^{1-a_{ij}-m} e_j e_i^m = 0, \quad \quad \sum_{m=0}^{1-a_{ij}}(-1)^m \binom{1-a_{ij}}{m}_{q^{d_i}} \; f_i^{1-a_{ij}-m} f_j f_i^m = 0,
\end{gather*}
for the abelian group $K := \langle k_1, \dots, k_{t}\rangle \cong \mathbb{Z}_l^{\times t}$. 

The above datum also defines a Nichols algebra of diagonal type by setting
$q_{ij}:=q^{d_ia_{ij}}$. The braiding given by $\qs$ is of \emph{Cartan type}, i.e. satisfies 
$q_{ij}q_{ji}=q_{ii}^{a_{ij}},$ for all $i,j=1,\ldots,t$. Moreover, the associated Nichols algebra $\BB_\qs$ is isomorphic to $u_q(\fr{n}_+)$, the positive part of the small quantum group associated to $\fr{g}$, and the braided Drinfeld double $\Drin_{K^*}(\BB_\qs,\BB_\qs^*)$ is isomorphic to $u_q(\fr{g})$, both via
$$\gamma_i=\ov{\gamma}_i \mapsto k_i, \quad x_i \mapsto e_i, \quad y_i\mapsto k_if_i(q^{-d_i}-q^{d_i}).$$
See, e.g., \cite{Lusztig}, \cite{AS}*{Theorem 4.3}, and references therein, for the isomorphism of $\BB_\qs$ and $u_q(\fr{n}_+)$, along with \cite{Som}*{Section 5.10}, \cite{Maj99}*{Proposition~4.3}, \cite{L18}*{Theorem~4.9} for the isomorphisms of the braided Drinfeld doubles.

\smallskip

We obtain that the category $\lmod{\Drin_{K^*}(\BB_\qs,\BB_\qs^*)} \simeq \lmod{u_q(\fr{g})}$ is modular by applying Proposition \ref{cor:UBqmod} as follows. First,  the pairing $r\colon G\times G\to \Bbbk$ obtained from $\qs$ as in Notation \ref{not:Cq-etc} is non-degenerate using a computation as in \cite{Ro}*{proof of Proposition 3.5},  assuming that the determinant of $(d_ia_{ij})$ is coprime to $l$.
This implies that $G$ is isomorphic to the group $ \langle\gamma_1,\ldots,\gamma_r \rangle\subseteq \Bbbk G^*$. Further, $r$ is symmetric and the associated bilinear form $b$ is given as its square. Thus, as all $q^{d_ia_{ij}}$ are primitive $l$-th roots of unity, with $l$ odd, the same holds for $q^{2d_ia_{ij}}$ and $b$ is also non-degenerate. 
Therefore, Proposition \ref{cor:UBqmod}(i) holds.
Moreover, the proof of \cite{Bur}*{Theorem~5.4} contains a computation that shows that  Proposition \ref{cor:UBqmod}(ii) holds for this class of examples.\footnote{In fact, Burciu denotes $\delta=\zeta, h=a, \chi_i(g)=r(g_i,g)$ and verifies the required equation $\delta^{-1}(g_i)\chi(h)=\chi_i(g_i)^{-1}$ using the Lie theoretic computation that, writing the $j$-th positive root as $\beta_j=\sum_{s=1}^t c_{js}\alpha_s$, we have $\sum_j\sum_s a_{is}c_{js}=2$.} 
Therefore, Proposition \ref{cor:UBqmod} implies that $\lmod{\Drin_{K^*}(\BB_\qs,\BB_\qs^*)}$ is a modular tensor category. 

\smallskip

Continuing Example~\ref{ex:decomp1}, we have a decomposition of modular categories:
$$\lmod{\Drin(u_q(\fr{b}_+))}\simeq \cB_{\qs}\boxtimes \lmod{u_q(\fr{g})},$$
where $u_q({\fr{b}}_+)$ is the positive Borel part of $u_q(\fr{g})$ generated by the $e_i$ and $k_i$.
\end{example}

\medskip

Finally, we produce an   example of a relative monoidal center that gives a modular category that is \emph{not} of the form $\lmod{u_q(\fr{g})}$. It consists of modules over a more general type of quantum group, namely, modules over the braided Drinfeld double of a Nichols algebra that is \emph{not} of Cartan type.

\begin{example}\label{expl-new}
Let $q \in \Bbbk$ be a primitive $2n$-th root of unity, for $n\geq 1$ an odd integer. We denote by $G$ the abelian group $\langle g_1,g_2\rangle=\mZ_{2n} \times \mZ_{2n}$. 
Consider the Nichols algebra $\BB_\qs$ of diagonal type determined by $\qs=(q_{ij})$, with
$$q_{11}=q_{22}=-1,\qquad q_{12}=1, \qquad q_{21}=q,$$
as in Section \ref{sec:Nichols}. This Nichols algebra appears in the classification of \cite{Hec2}*{Table 1, Row 2}. It has by \cite{Kha} a PBW basis given by the set
$\left\{x_1^{a_1}x_2^{a_2}x_{12}^{a_{12}}\; \big|\; 0\leq a_1,a_2< 2, \;0\leq a_{12}< 2n\right\},$
where $x_{12}:=x_1x_2-x_2x_1$, and is thus  $8n$-dimensional. Note that $\BB_\qs$  is generated by $x_1,x_2$ subject to the relations 
$$x_1^2=0,\qquad x_2^2=0, \qquad x_{12}^n=0,$$
see \cite{AA}*{Section 5.1.11}, and is one of the Nichols algebras of super type $A(1|1)$.

In this example, the symmetric bilinear form $b$ from Notation \ref{not:Cq-etc} is given by $b(g_i,g_j) = q_{ji}q_{ij}$, so
$$b(g_1,g_1^{c_1}g_2^{c_2})=b(g_1,g_1)^{c_1}b(g_1,g_2)^{c_2}=q^{c_2},\quad b(g_2,g_1^{c_1}g_2^{c_2})=b(g_2,g_1)^{c_1}b(g_2,g_2)^{c_2}=q^{c_1},~~ \text{for all $c_1,c_2$.}$$
 Hence, $g_1^{c_1}g_2^{c_2}$ is in the radical of $b$ if and only if $c_1,c_2=0 \mod 2n$ since $q$ is a primitive $2n$-th root of unity. Thus, $b$
is non-degenerate and Proposition~\ref{cor:UBqmod}(i) holds by Lemma \ref{lem:Bq-nondeg}.

\smallskip

The Hopf algebra $\Drin_{K^*}(\BB_\qs,\BB_\qs^*)$ may be presented as the Hopf algebra generated by $x_i, y_i,$ and $k_i:=\gamma_i\ov{\gamma}_i$ for $i=1,2$, subject to relations, for $i,j=1,2$, $i\neq j$,
\begin{gather*}
    k_i x_i=x_i k_i, \qquad k_i y_i=y_i k_i,\qquad k_ix_j=qx_j k_i, \qquad k_iy_j=q^{-1}y_j k_i, \qquad x_iy_j+y_jx_i=\delta_{i,j}(1-k_i),\\
x_i^2=y_i^2=0,\qquad k_i^{2n}=1,\qquad  (x_1x_2-x_2x_1)^{2n}=(y_2y_1-y_1y_2)^{2n}=0,\\
    \Delta(x_1)=x_1\otimes 1+k_2^n\otimes x_1, \qquad \Delta(x_2)=x_2\otimes 1+ k_1^nk_2\otimes x_2,\\
    \Delta(y_1)=y_1\otimes 1+k_1k_2^n\otimes y_1, \qquad \Delta(y_2)=y_2\otimes 1+ k_1^n\otimes y_2.
\end{gather*}
Here, we use that $\gamma_1=k_2^n$,\; $\gamma_2=k_1^nk_2$,\; $\ov{\gamma}_1=k_1k_2^n$,\; and $\ov{\gamma}_2=k_1^n$. Note that $\langle \gamma_1,\gamma_2 \rangle$ and $\langle \ov{\gamma}_1,\ov{\gamma}_2 \rangle$ are both proper subgroups of $\langle k_1,k_2 \rangle$.

\smallskip

The top $\mZ$-degree element $x_{{\bf i}_\ell}$ is $x_1x_2x_{12}^{2n-1}$ and has $G$-degree ${\bf i}_\ell=(2n,2n)=0\in \Lambda=\mZ_{2n}\times \mZ_{2n}$. There are four pairs ${\bf j}=(j_1,j_2)$, ${\bf a}=(a_1,a_2)\in \Lambda$ satisfying the conditions from Proposition~\ref{cor:UBqmod}(ii):
\begin{center}
(1)\; ${\bf j}=(0,0)$, ${\bf a}=(n,n)$; \qquad \qquad
(2)\; ${\bf j}=(n,0)$, ${\bf a}=(n,0)$;\\
\,(3)\; ${\bf j}=(0,n)$, ${\bf a}=(0,0)$; \qquad \qquad
\,(4)\; ${\bf j}=(n,n)$, ${\bf a}=(0,n)$;
\end{center}
It is clear that in all cases (1)--(4), $2{\bf j}=0={\bf i}_\ell$ and that $k_{\bf a}=k_1^{a_1}k_2^{a_2}$ squares to $g_H=\gamma_0=1$, so that $b(g_i,g_{\bf a})^2=r(g_i,g_{{\bf i}_\ell})=1$ for all $i=1,2$. The remaining condition of  Proposition \ref{cor:UBqmod}(ii) is verified using explicit computation. 
In Case (1), for instance, we compute that
\begin{align*}
   r(g_{\bf j},g_1)b(g_1,g_{\bf a})= r(g_{(0,0)},g_1)b(g_1,g_{(n,n)})&=q_{11}^{2n}(q_{12}q_{21})^n=(-1)^{2n}q^n=-1=q_{11}^{-1},
\end{align*}
using $q^n=-1$. Similarly,
\begin{align*}
r(g_{\bf j},g_1)b(g_1, g_{\bf a}) =r(g_{(0,0)},g_2) b(g_2,g_{(n,n)})&=q_{22}^{2n}(q_{21}q_{12})^n=(-1)^{2n}q^n=-1=q_{22}^{-1}.
\end{align*}
Thus, the set $\mathsf{Sqrt}_{\cC}(D,\xi_D)$ contains four distinct elements, which yields four different ribbon structures on  $\lmod{\Drin_{K^*}(\BB_\qs,\BB_\qs^*)}$ by the proof of Theorem~\ref{thm:ZBCmodular} (using Theorem~\ref{thm:Shimizuribbon}). In each case, this category is modular by 
Proposition~\ref{cor:UBqmod}.

\smallskip

Moreover, to study the sphericality of the category $\cC:=\lmod{\BB_\qs}(\cB_\qs)$, we apply Remark~\ref{rmk:UBqmod}(3), and show that only  Case~(1) yields a spherical structure. For this, we set ${\bf b}=(1,0)$, ${\bf c}=(0,1)$, and get that $a=\gamma_{\bf b}\ov{\gamma}_{\bf c}=\gamma_1\ov{\gamma}_2=k_1^nk_2^n$. Then $a^2=k_1^{2n}k_2^{2n}=1$ and the equations $r(g_i,g_{\bf b})r(g_{\bf c},g_i)=q_{ii}^{-1}$, for $i=1,2$ follow similarly to above, thus satisfying the conditions in Remark~\ref{rmk:UBqmod}(3). Note that $a$ here is the same as the one obtained in Case (1); thus, in this case, $\cC$ is spherical in the sense of Definition~\ref{def:spherical}.
Finally, Cases (2)--(4) do not yield spherical structures on $\cC$. Indeed, by Proposition~\ref{prop:Hspherical}, having $a \in \mathsf{SPiv}(H)$ is equivalent to $(\zeta, a)$ belonging to the set of Theorem~\ref{thm:KR} with $\zeta = \varepsilon$. The pairs $({\bf j}, {\bf a})$ above correspond to pairs $(\zeta, a)$ in Theorem~\ref{thm:KR}, with $a = k_{\bf a}$ and $\zeta$  obtained from ${\bf j}$ via $\zeta(x\otimes \delta_{\bf i})=\varepsilon(x)\delta_{{\bf i},{\bf j}}$. To get that $\zeta = \varepsilon$, we need that ${\bf j} = 0$ which only occurs in Case~(1). 
\end{example}


\section*{Acknowledgements} 

The authors would like to thank the referees for their detailed suggestions that greatly improved the exposition of this work. The authors also thank Iv\'an Angiono, Alexei Davydov, J\"urgen Fuchs, and Christoph Schweigert for interesting exchanges related to this ongoing project, 
Kenichi Shimizu for helpful exchanges about his work \cite{Shi2} discussed in Section~\ref{sec:centerribbon}.

R.~Laugwitz was partially supported by an AMS-Simons travel grant and a Nottingham Research Fellowship. C. Walton was partially supported by a research fellowship from the Alfred P. Sloan foundation, and by the US National Science Foundation grant \#DMS-1903192. Part of the work for this project was carried out during a visit of R.~L. to UIUC; the hospitality of the hosting institution is gratefully acknowledged.

\bigskip

\bibliography{lm-qea-bib}

\begin{bibdiv}
\begin{biblist}

\bib{AA}{article}{
      author={Andruskiewitsch, Nicol\'{a}s},
      author={Angiono, Iv\'{a}n},
       title={On finite dimensional {N}ichols algebras of diagonal type},
        date={2017},
        ISSN={1664-3607},
     journal={Bull. Math. Sci.},
      volume={7},
      number={3},
       pages={353\ndash 573},
         url={https://doi.org/10.1007/s13373-017-0113-x},
}

\bib{Ang1}{article}{
      author={Angiono, Iv\'{a}n},
       title={On {N}ichols algebras of diagonal type},
        date={2013},
        ISSN={0075-4102},
     journal={J. Reine Angew. Math.},
      volume={683},
       pages={189\ndash 251},
         url={https://doi.org/10.1515/crelle-2011-0008},
}

\bib{Ang2}{article}{
      author={Angiono, Iv\'{a}n~Ezequiel},
       title={A presentation by generators and relations of {N}ichols algebras
  of diagonal type and convex orders on root systems},
        date={2015},
        ISSN={1435-9855},
     journal={J. Eur. Math. Soc. (JEMS)},
      volume={17},
      number={10},
       pages={2643\ndash 2671},
         url={https://doi.org/10.4171/JEMS/567},
}

\bib{AS}{incollection}{
      author={Andruskiewitsch, Nicol\'{a}s},
      author={Schneider, Hans-J\"{u}rgen},
       title={Pointed {H}opf algebras},
        date={2002},
   booktitle={New directions in {H}opf algebras},
      series={Math. Sci. Res. Inst. Publ.},
      volume={43},
   publisher={Cambridge Univ. Press, Cambridge},
       pages={1\ndash 68},
         url={https://doi.org/10.2977/prims/1199403805},
}

\bib{Bes}{article}{
      author={Bespalov, Yu.~N.},
       title={Crossed modules and quantum groups in braided categories},
        date={1997},
     journal={Appl. Categ. Structures},
      volume={5},
      number={2},
       pages={155\ndash 204},
}

\bib{BK}{book}{
      author={Bakalov, Bojko},
      author={Kirillov, Alexander, Jr.},
       title={Lectures on tensor categories and modular functors},
      series={University Lecture Series},
   publisher={American Mathematical Society, Providence, RI},
        date={2001},
      volume={21},
        ISBN={0-8218-2686-7},
}

\bib{Bur}{article}{
      author={Burciu, Sebastian},
       title={A class of {D}rinfeld doubles that are ribbon algebras},
        date={2008},
        ISSN={0021-8693},
     journal={J. Algebra},
      volume={320},
      number={5},
       pages={2053\ndash 2078},
         url={https://doi.org/10.1016/j.jalgebra.2008.05.021},
}

\bib{BW}{article}{
      author={Barrett, John~W.},
      author={Westbury, Bruce~W.},
       title={Spherical categories},
        date={1999},
        ISSN={0001-8708},
     journal={Adv. Math.},
      volume={143},
      number={2},
       pages={357\ndash 375},
         url={https://doi.org/10.1006/aima.1998.1800},
}

\bib{Del}{incollection}{
      author={Deligne, P.},
       title={Cat\'{e}gories tannakiennes},
        date={1990},
   booktitle={The {G}rothendieck {F}estschrift, {V}ol. {II}},
      series={Progr. Math.},
      volume={87},
   publisher={Birkh\"{a}user Boston, Boston, MA},
       pages={111\ndash 195},
}

\bib{DGNO}{article}{
      author={Drinfeld, Vladimir},
      author={Gelaki, Shlomo},
      author={Nikshych, Dmitri},
      author={Ostrik, Victor},
       title={On braided fusion categories. {I}},
        date={2010},
        ISSN={1022-1824},
     journal={Selecta Math. (N.S.)},
      volume={16},
      number={1},
       pages={1\ndash 119},
         url={https://doi.org/10.1007/s00029-010-0017-z},
}

\bib{DMNO}{article}{
      author={Davydov, Alexei},
      author={M\"{u}ger, Michael},
      author={Nikshych, Dmitri},
      author={Ostrik, Victor},
       title={The {W}itt group of non-degenerate braided fusion categories},
        date={2013},
        ISSN={0075-4102},
     journal={J. Reine Angew. Math.},
      volume={677},
       pages={135\ndash 177},
}

\bib{DNO}{article}{
      author={Davydov, Alexei},
      author={Nikshych, Dmitri},
      author={Ostrik, Victor},
       title={On the structure of the {W}itt group of braided fusion
  categories},
        date={2013},
        ISSN={1022-1824},
     journal={Selecta Math. (N.S.)},
      volume={19},
      number={1},
       pages={237\ndash 269},
         url={https://doi.org/10.1007/s00029-012-0093-3},
}

\bib{DGGPR}{article}{
      author={De~Renzi, Marco},
      author={Gainutdinov, Azat~M},
      author={Geer, Nathan},
      author={Patureau-Mirand, Bertrand},
      author={Runkel, Ingo},
       title={3-dimensional {TQFT}s from non-semisimple modular categories},
        date={2019},
     journal={arXiv preprint arXiv:1912.02063v2 [math.GT]},
}

\bib{DSS}{article}{
      author={Douglas, Christopher~L.},
      author={Schommer-Pries, Christopher},
      author={Snyder, Noah},
       title={Dualizable tensor categories},
        date={2018},
     journal={arXiv preprint arXiv:1312.7188v2 [math.QA]},
}

\bib{EGNO}{book}{
      author={Etingof, Pavel},
      author={Gelaki, Shlomo},
      author={Nikshych, Dmitri},
      author={Ostrik, Victor},
       title={Tensor categories},
      series={Mathematical Surveys and Monographs},
   publisher={American Mathematical Society, Providence, RI},
        date={2015},
      volume={205},
}

\bib{ENO}{article}{
      author={Etingof, Pavel},
      author={Nikshych, Dmitri},
      author={Ostrik, Viktor},
       title={An analogue of {R}adford's {$S^4$} formula for finite tensor
  categories},
        date={2004},
        ISSN={1073-7928},
     journal={Int. Math. Res. Not.},
      number={54},
       pages={2915\ndash 2933},
         url={https://doi.org/10.1155/S1073792804141445},
}

\bib{FSS}{article}{
      author={Fuchs, Jürgen},
      author={Schaumann, Gregor},
      author={Schweigert, Christoph},
       title={A modular functor from state sums for finite tensor categories
  and their bimodules},
        date={2019},
     journal={Arxiv preprint arXiv:1911.06214 [math.QA]},
}

\bib{Gannon}{article}{
      author={Gannon, Terry},
       title={Modular data: the algebraic combinatorics of conformal field
  theory},
        date={2005},
        ISSN={0925-9899},
     journal={J. Algebraic Combin.},
      volume={22},
      number={2},
       pages={211\ndash 250},
         url={https://doi-org.ezproxy.rice.edu/10.1007/s10801-005-2514-2},
}

\bib{GLO}{article}{
      author={Gainutdinov, Azat~M.},
      author={Lentner, Simon},
      author={Ohrmann, Tobias},
       title={Modularization of small quantum groups},
        date={2018},
     journal={arXiv preprint arXiv:1809.02116v2 [math.QA]},
}

\bib{Hec1}{article}{
      author={Heckenberger, I.},
       title={The {W}eyl groupoid of a {N}ichols algebra of diagonal type},
        date={2006},
        ISSN={0020-9910},
     journal={Invent. Math.},
      volume={164},
      number={1},
       pages={175\ndash 188},
         url={https://doi.org/10.1007/s00222-005-0474-8},
}

\bib{Hec2}{article}{
      author={Heckenberger, I.},
       title={Classification of arithmetic root systems},
        date={2009},
        ISSN={0001-8708},
     journal={Adv. Math.},
      volume={220},
      number={1},
       pages={59\ndash 124},
         url={https://doi.org/10.1016/j.aim.2008.08.005},
}

\bib{Hec4}{article}{
      author={Heckenberger, I.},
       title={Lusztig isomorphisms for {D}rinfel\cprime d doubles of
  bosonizations of {N}ichols algebras of diagonal type},
        date={2010},
        ISSN={0021-8693},
     journal={J. Algebra},
      volume={323},
      number={8},
       pages={2130\ndash 2182},
         url={https://doi.org/10.1016/j.jalgebra.2010.02.013},
}

\bib{HLZ}{incollection}{
      author={Huang, Yi-Zhi},
      author={Lepowsky, James},
      author={Zhang, Lin},
       title={Logarithmic tensor category theory for generalized modules for a
  conformal vertex algebra, {I}: introduction and strongly graded algebras and
  their generalized modules},
        date={2014},
   booktitle={Conformal field theories and tensor categories},
      series={Math. Lect. Peking Univ.},
   publisher={Springer, Heidelberg},
       pages={169\ndash 248},
}

\bib{HS}{book}{
      author={Heckenberger, I.},
      author={Schneider, Hans-J\"urgen},
       title={Hopf {A}lgebras and {R}oot {S}ystems},
      series={Mathematical Surveys and Monographs},
   publisher={American Mathematical Society, Providence, RI},
        date={2020},
      volume={247},
        ISBN={978-1-4704-5680-1},
}

\bib{Hua}{article}{
      author={Huang, Yi-Zhi},
       title={Vertex operator algebras, the {V}erlinde conjecture, and modular
  tensor categories},
        date={2005},
        ISSN={0027-8424},
     journal={Proc. Natl. Acad. Sci. USA},
      volume={102},
      number={15},
       pages={5352\ndash 5356},
         url={https://doi.org/10.1073/pnas.0409901102},
}

\bib{Kha}{article}{
      author={Kharchenko, V.~K.},
       title={A quantum analogue of the {P}oincar\'{e}-{B}irkhoff-{W}itt
  theorem},
        date={1999},
        ISSN={0373-9252},
     journal={Algebra Log.},
      volume={38},
      number={4},
       pages={476\ndash 507, 509},
         url={https://doi.org/10.1007/BF02671731},
}

\bib{KL}{book}{
      author={Kerler, Thomas},
      author={Lyubashenko, Volodymyr~V.},
       title={Non-semisimple topological quantum field theories for 3-manifolds
  with corners},
      series={Lecture Notes in Mathematics},
   publisher={Springer-Verlag, Berlin},
        date={2001},
      volume={1765},
        ISBN={3-540-42416-4},
}

\bib{KLM-2001}{article}{
      author={Kawahigashi, Yasuyuki},
      author={Longo, Roberto},
      author={M\"{u}ger, Michael},
       title={Multi-interval subfactors and modularity of representations in
  conformal field theory},
        date={2001},
        ISSN={0010-3616},
     journal={Comm. Math. Phys.},
      volume={219},
      number={3},
       pages={631\ndash 669},
         url={https://doi-org.ezproxy.rice.edu/10.1007/PL00005565},
}

\bib{KR93}{article}{
      author={Kauffman, Louis~H.},
      author={Radford, David~E.},
       title={A necessary and sufficient condition for a finite-dimensional
  {D}rinfel\cprime d double to be a ribbon {H}opf algebra},
        date={1993},
        ISSN={0021-8693},
     journal={J. Algebra},
      volume={159},
      number={1},
       pages={98\ndash 114},
         url={https://doi.org/10.1006/jabr.1993.1148},
}

\bib{L17}{article}{
      author={Laugwitz, Robert},
       title={Comodule algebras and 2-cocycles over the (braided) {D}rinfeld
  double},
        date={2019},
        ISSN={0219-1997},
     journal={Commun. Contemp. Math.},
      volume={21},
      number={4},
       pages={1850045, 46},
         url={https://doi.org/10.1142/S0219199718500451},
}

\bib{L18}{article}{
      author={Laugwitz, Robert},
       title={The relative monoidal center and tensor products of monoidal
  categories},
        date={2020},
        ISSN={0219-1997},
     journal={Commun. Contemp. Math.},
      volume={22},
      number={8},
       pages={1950068, 53},
         url={https://doi.org/10.1142/S0219199719500688},
}

\bib{LMSS}{article}{
      author={Lentner, Simon},
      author={Mierach, Svea~Nora},
      author={Schweigert, Christoph},
      author={Sommerhaeuser, Yorck},
       title={Hochschild cohomology, modular tensor categories, and mapping
  class groups},
        date={2020},
     journal={ArXiv preprint arXiv:2003.06527 [math.QA]},
}

\bib{LO}{article}{
      author={Lentner, Simon},
      author={Ohrmann, Tobias},
       title={Factorizable {$R$}-matrices for small quantum groups},
        date={2017},
     journal={SIGMA Symmetry Integrability Geom. Methods Appl.},
      volume={13},
       pages={Paper No. 076, 25},
         url={https://doi.org/10.3842/SIGMA.2017.076},
}

\bib{Lusztig}{book}{
      author={Lusztig, George},
       title={Introduction to quantum groups},
      series={Modern Birkh\"{a}user Classics},
   publisher={Birkh\"{a}user/Springer, New York},
        date={2010},
        note={Reprint of the 1994 edition},
}

\bib{LW}{article}{
      author={{Laugwitz}, Robert},
      author={{Walton}, Chelsea},
       title={{Braided commutative algebras over quantized enveloping
  algebras}},
        date={2020},
     journal={Transformation Groups},
}

\bib{Lyu}{article}{
      author={Lyubashenko, Volodymyr~V.},
       title={Invariants of {$3$}-manifolds and projective representations of
  mapping class groups via quantum groups at roots of unity},
        date={1995},
        ISSN={0010-3616},
     journal={Comm. Math. Phys.},
      volume={172},
      number={3},
       pages={467\ndash 516},
         url={http://projecteuclid.org/euclid.cmp/1104274312},
}

\bib{Mue2}{article}{
      author={M\"{u}ger, Michael},
       title={From subfactors to categories and topology. {II}. {T}he quantum
  double of tensor categories and subfactors},
        date={2003},
        ISSN={0022-4049},
     journal={J. Pure Appl. Algebra},
      volume={180},
      number={1-2},
       pages={159\ndash 219},
         url={https://doi.org/10.1016/S0022-4049(02)00248-7},
}

\bib{Mue}{article}{
      author={M\"{u}ger, Michael},
       title={On the structure of modular categories},
        date={2003},
        ISSN={0024-6115},
     journal={Proc. London Math. Soc. (3)},
      volume={87},
      number={2},
       pages={291\ndash 308},
         url={https://doi.org/10.1112/S0024611503014187},
}

\bib{Majid}{book}{
      author={Majid, Shahn},
       title={Foundations of quantum group theory},
   publisher={Cambridge University Press, Cambridge},
        date={2000},
        ISBN={0-521-64868-8},
         url={http://dx.doi.org/10.1017/CBO9780511613104},
        note={Paperback Edition, originally published 1995},
}

\bib{Maj99}{article}{
      author={Majid, Shahn},
       title={Double-bosonization of braided groups and the construction of
  {$U_q(\germ g)$}},
        date={1999},
     journal={Math. Proc. Cambridge Philos. Soc.},
      volume={125},
      number={1},
       pages={151\ndash 192},
}

\bib{MoS}{article}{
      author={Moore, Gregory},
      author={Seiberg, Nathan},
       title={Classical and quantum conformal field theory},
        date={1989},
        ISSN={0010-3616},
     journal={Comm. Math. Phys.},
      volume={123},
      number={2},
       pages={177\ndash 254},
         url={http://projecteuclid.org/euclid.cmp/1104178762},
}

\bib{Negron}{article}{
      author={Negron, Cris},
       title={Log-{M}odular {Q}uantum {G}roups at {E}ven {R}oots of {U}nity and
  the {Q}uantum {F}robenius {I}},
        date={2021},
        ISSN={0010-3616},
     journal={Comm. Math. Phys.},
      volume={382},
      number={2},
       pages={773\ndash 814},
         url={https://doi.org/10.1007/s00220-021-04012-2},
}

\bib{Rad}{book}{
      author={Radford, David~E.},
       title={Hopf algebras},
      series={Series on Knots and Everything},
   publisher={World Scientific Publishing Co. Pte. Ltd., Hackensack, NJ},
        date={2012},
      volume={49},
        ISBN={978-981-4335-99-7; 981-4335-99-1},
}

\bib{Rad76}{article}{
      author={Radford, David~E.},
       title={The order of the antipode of a finite dimensional {H}opf algebra
  is finite},
        date={1976},
        ISSN={0002-9327},
     journal={Amer. J. Math.},
      volume={98},
      number={2},
       pages={333\ndash 355},
         url={https://doi-org.ezproxy.rice.edu/10.2307/2373888},
}

\bib{Ro}{incollection}{
      author={Rosso, Marc},
       title={Quantum groups at a root of 1 and tangle invariants},
        date={1993},
      volume={7},
       pages={3715\ndash 3726},
         url={https://doi.org/10.1142/S0217979293003462},
        note={Yang-Baxter equations in Paris (1992)},
}

\bib{Ros}{book}{
      author={Rosenberg, Alexander~L.},
       title={Noncommutative algebraic geometry and representations of
  quantized algebras},
      series={Mathematics and its Applications},
   publisher={Kluwer Academic Publishers Group, Dordrecht},
        date={1995},
      volume={330},
        ISBN={0-7923-3575-9},
  url={https://doi-org.proxy2.library.illinois.edu/10.1007/978-94-015-8430-2},
}

\bib{RSW}{article}{
      author={Rowell, Eric},
      author={Stong, Richard},
      author={Wang, Zhenghan},
       title={On classification of modular tensor categories},
        date={2009},
        ISSN={0010-3616},
     journal={Comm. Math. Phys.},
      volume={292},
      number={2},
       pages={343\ndash 389},
         url={https://doi-org.ezproxy.rice.edu/10.1007/s00220-009-0908-z},
}

\bib{RT90}{article}{
      author={Reshetikhin, N.~Yu.},
      author={Turaev, V.~G.},
       title={Ribbon graphs and their invariants derived from quantum groups},
        date={1990},
        ISSN={0010-3616},
     journal={Comm. Math. Phys.},
      volume={127},
      number={1},
       pages={1\ndash 26},
         url={http://projecteuclid.org/euclid.cmp/1104180037},
}

\bib{Shi2}{article}{
      author={Shimizu, Kenichi},
       title={Ribbon structures of the {D}rinfeld center of a finite tensor
  category},
        date={2018},
     journal={arXiv preprint arXiv:1707.09691v2 [math.QA]},
}

\bib{Shi1}{article}{
      author={Shimizu, Kenichi},
       title={Non-degeneracy conditions for braided finite tensor categories},
        date={2019},
        ISSN={0001-8708},
     journal={Adv. Math.},
      volume={355},
       pages={106778, 36},
         url={https://doi.org/10.1016/j.aim.2019.106778},
}

\bib{Shi3}{inproceedings}{
      author={Shimizu, Kenichi},
       title={Recent developments of the categorical {V}erlinde formula},
        date={2019},
   booktitle={Proceedings of the {M}eeting for {S}tudy of {N}umber {T}heory,
  {H}opf {A}lgebras and {R}elated {T}opics},
   publisher={Yokohama Publ., Yokohama},
       pages={197\ndash 222},
}

\bib{Shum}{article}{
      author={Shum, Mei~Chee},
       title={Tortile tensor categories},
        date={1994},
        ISSN={0022-4049},
     journal={J. Pure Appl. Algebra},
      volume={93},
      number={1},
       pages={57\ndash 110},
         url={https://doi.org/10.1016/0022-4049(92)00039-T},
}

\bib{Som}{article}{
      author={Sommerh\"{a}user, Yorck},
       title={Deformed enveloping algebras},
        date={1996},
     journal={New York J. Math.},
      volume={2},
       pages={35\ndash 58, electronic},
         url={http://nyjm.albany.edu:8000/j/1996/2_35.html},
}

\bib{Turaev-1992}{incollection}{
      author={Turaev, Vladimir~G.},
       title={Modular categories and {$3$}-manifold invariants},
        date={1992},
      volume={6},
       pages={1807\ndash 1824},
         url={https://doi-org.ezproxy.rice.edu/10.1142/S0217979292000876},
        note={Topological and quantum group methods in field theory and
  condensed matter physics},
}

\bib{TV}{book}{
      author={Turaev, Vladimir},
      author={Virelizier, Alexis},
       title={Monoidal categories and topological field theory},
      series={Progress in Mathematics},
   publisher={Birkh\"{a}user/Springer, Cham},
        date={2017},
      volume={322},
}

\end{biblist}
\end{bibdiv}
\bibliographystyle{amsrefs}

\end{document}